\documentclass{amsart}
\usepackage{cite}
\usepackage[T1]{fontenc}
\usepackage{indentfirst}
\usepackage{amssymb}
\usepackage{amsmath}
\usepackage{graphics}
\usepackage{epsfig}
\usepackage{xcolor}

\usepackage{MnSymbol,wasysym}

\usepackage[utf8]{inputenc}
\usepackage{latexsym}
\usepackage{graphicx}

\usepackage[normalem]{ulem}
\usepackage{soul}

\catcode`@=11 \@addtoreset{equation}{section}
\renewcommand\theequation{\thesection.\@arabic\c@equation}
\catcode`@=12

\newtheorem{theorem}{Theorem}[section]
\newtheorem{remark}[theorem]{Remark}
\newtheorem{claim}[theorem]{Claim}

\expandafter\chardef\csname pre amssym.def
at\endcsname=\the\catcode`\@ \catcode`\@=11
\def\undefine#1{\let#1\undefined}
\def\newsymbol#1#2#3#4#5{\let\next@\relax
 \ifnum#2=\@ne\let\next@\msafam@\else
 \ifnum#2=\tw@\let\next@\msbfam@\fi\fi
 \mathchardef#1="#3\next@#4#5}
\def\mathhexbox@#1#2#3{\relax
 \ifmmode\mathpalette{}{\m@th\mathchar"#1#2#3}%
 \else\leavevmode\hbox{$\m@th\mathchar"#1#2#3$}\fi}
\def\hexnumber@#1{\ifcase#1 0\or 1\or 2\or 3\or 4\or 5\or 6\or 7\or 8\or
 9\or A\or B\or C\or D\or E\or F\fi}

\font\teneufm=eufm10 \font\seveneufm=eufm7 \font\fiveeufm=eufm5
\newfam\eufmfam
\textfont\eufmfam=\teneufm \scriptfont\eufmfam=\seveneufm
\scriptscriptfont\eufmfam=\fiveeufm

\catcode`\@=\csname pre amssym.def at\endcsname

\newcommand{\ww}{\mbox{$\bf w $}}
\newcommand{\WW}{\mbox{$\bf W $}}
\newcommand{\bb}{\mbox{\boldmath $b$}}
\newcommand{\uu}{\mbox{\boldmath $u$}}
\newcommand{\vv}{\mbox{\boldmath $v$}}
\newcommand{\zz}{\mbox{\boldmath $z$}}
\newcommand{\xx}{\mbox{\boldmath $x$}}
\newcommand{\AAA}{\mathbb{A}}
\newcommand{\RR}{\mathbb{R}}
\newcommand{\LLL}{\mathbb{L}}

\newcommand{\norm}[1]{\Vert #1 \Vert}

\newcommand{\LL}{{L^2(\RR^3)}}
\newcommand{\LT}{{L^2(\RR^3)}}
\newcommand{\LP}{{L^p(\RR^3)}}

\newcommand{\Ez}{{\mathcal{E}}_{\zz}}
\newcommand{\Eu}{{\mathcal{E}}_{\uu}}
\newcommand{\Ew}{{\mathcal{E}}_{\ww}}
\newcommand{\QQ}{\mbox{\boldmath $Q$}}

\newtheorem{thm}{Theorem}[section]

\newtheorem{cor}[thm]{Corollary}

\newtheorem{rem}[thm]{Remark}
\newtheorem{lemma}{Lemma}
\begin{document}

\title[Asymptotic Inequalities]{High order asymptotic inequalities for some dissipative systems}

\author[P. Braz e Silva]{P. Braz e Silva}
\address[P. Braz e Silva]{Departamento de Matem\'atica. Universidade Federal de Pernambuco, 
	CEP 50740-560, Recife - PE. Brazil}
\email{pablo.braz@ufpe.br}

\author[R. H. Guterres]{R. H. Guterres}
\address[R. H. Guterres]{Departamento de Matem\'atica. Universidade Federal de Pernambuco, 
	CEP 50740-560, Recife - PE. Brazil}
\email{rguterres.mat@gmail.com}

\author[C. F. Perusato]{C. F. Perusato}
\address[C. F. Perusato]{Departamento de Matem\'atica. Universidade Federal de Pernambuco, 
 CEP 50740-560, Recife - PE. Brazil}
\email{cilon.perusato@ufpe.br}

\author[P. R. Zingano]{P. R.  Zingano}
\address[P. R. Zingano]{Departamento de Matem\'atica Pura e Aplicada, Instituto de Matem\'atica e Estatística. Universidade Federal do Rio Grande do Sul, CEP 91509-900, Porto Alegre - RS, Brazil}
\email{zingano@gmail.com}

\thanks{P. Braz e Silva was partially supported by CAPES-PRINT, Brazil, $\#$88881.311964/2018-01, 
CAPES-MATHAMSUD $\#$88881.520205/2020-0, and CNPq, Brazil, $\#$308758/2018-8 and $\#$432387/2018-8}

\thanks{R. H Guterres  acknowledges funding from Bolsa FACEPE Grant BFP-0067-1.01/19.} 

\thanks{C. Perusato was partially supported by CAPES--PRINT - 88881.311964/2018--01 and Propesq-UFPE - 08-2019 (Qualis A). }



\date{\today}

\begin{abstract}
We obtain some important fundamental inequalities concerning the long time behavior of high order derivatives for solutions of some dissipative systems in terms of their $L^2$ algebraic decay. Some of these inequalities have not been observed in the literature even for the fundamental Navier-Stokes equations.To illustrate this new approach, we derive bounds for the asymmetric incompressible fluids equations, where an improved inequalities established for the micro-rotational field. Finally, we show how this technique can be applied for other dissipative systems.        
\end{abstract}

\maketitle

\section{Introduction}
Asymptotic inequalities for dissipative systems are a topic of high interest, specially in the context of fluid mechanics models, such as the Navier-Stokes equations and related systems, like the asymmetric incompressible fluids equations (micropolar Navier-Stokes), among others.  Here, we derive some new inequalities for high order derivatives of solutions for such systems.
The method used here is, by itself, interesting, since it is a direct method based on energy inequalities. Even tough it is direct, it depends on very subtle and technical arguments, which may eventually be used in different contexts. To describe it, we show how to obtain bounds for solutions of the so called asymmetric, also known as micropolar, fluids system 
\begin{eqnarray*}
	\mbox{\boldmath $u$}_t
	+
	\mbox{\boldmath $u$} \cdot \nabla \mbox{\boldmath $u$}
	+
	\nabla \:\!{ \sf{p}}
	& = & 
	(\:\!\mu + \chi\:\!) \, \Delta \mbox{\boldmath $u$}
	\,+\,
	\chi \, \nabla \times {\bf w}, \\
	{\bf w}_t
	+
	\mbox{\boldmath $u$} \cdot \nabla \mbox{\bf w}
	& = & 
	\gamma \;\! \, \Delta \mbox{\bf w}
	\,+\,\kappa\, \nabla \:\!(\:\!\nabla \cdot\;{\bf w} \:\!)
	\,+\,
	\chi \, \nabla \times\, \mbox{\boldmath $u$}
	\,-\, 2 \, \chi \, {\bf w}, \\
	\nabla \cdot \mbox{\boldmath $u$}(\cdot,t) &  =  & 0,
\end{eqnarray*}
with suitable initial data. 
We note that the results obtained here
for the linear velocity 
$\mathbf{u}$ in system above above also hold for solutions of the classical Navier-Stokes equations. 
In particular, as far as we know, the results on the large time dynamics given in Theorem \ref{thm2} below are new even 
for solutions of the Navier-Stokes equations.


In 1984, Kato \cite{MR760047} and Masuda \cite{MR767409} showed that solutions $\mathbf{u}$ of the Navier-Stokes 
equations satisfy $ \|\uu(\cdot,t)\|_{L^2(\mathbb{R}^{3})} \to 0 $ as $ t \to \infty $, finally solving an open problem left 
by Leray \cite{MR1555394}. Since then, several works have made considerable progress in this direction (see e.g \cite{MR1749867, MR1312701, MR837929, MR1396285,  MR775190}). In particular, the well-known Fourier Splitting method, developed by M. E. Schonbek \cite{MR571048,MR837929,MR775190}, is a powerful tool to prove that the $L^2$ norm of solutions for viscous conservation laws and for Navier-Stokes equations decay with algebraic rate depending on the initial data. 
Recently, Hagstrom et al. \cite{MR4021907} derived an asymptotic inequality for arbitrary order spatial derivatives of Leray solutions for the incompressible Navier-Stokes equations in $\RR^n$ with $n \leq 4$. This technique may also be extended to dissipative systems with linear damping, such as the asymmetric fluids system, as we show in details here (see Theorem \ref{thm1}). 
The bounds for the linear velocity $\mathbf{u}$ are the same as for solutions of classical Navier-Stokes equations obtained in \cite{MR4021907}. The bounds for the rotational velocity $\mathbf{w}$ are stronger, as can be seen 
in Theorems \ref{thm1} and \ref{thm2}. 

 Firstly introduced by Eringen in \cite{MR0204005}, the asymmetric fluids model is a generalization of the classical Navier-Stokes equations that takes into account the microstructure of the fluid particles. It describes flows of fluids whose particles undergo translations and rotations as well. Asymmetric fluids flows are governed by the equations
\begin{subequations}\label{micropolar}
	\begin{equation}\label{uu_equation}
	\mbox{\boldmath $u$}_t
	\;\!+\,
	\mbox{\boldmath $u$} \cdot \nabla \mbox{\boldmath $u$}
	\,+\;\!
	\nabla \:\!{ \sf{p}}
	\;=\;
	(\:\!\mu + \chi\:\!) \, \Delta \mbox{\boldmath $u$}
	\,+\,
	\chi \, \nabla \times {\bf w},
	\end{equation}
	\begin{equation}\label{ww_equation}
	{\bf w}_t
	\;\!+\,
	\mbox{\boldmath $u$} \cdot \nabla \mbox{\bf w}
	\;=\;
	\gamma \;\! \, \Delta \mbox{\bf w}
	\,+\,\kappa\, \nabla \:\!(\:\!\nabla \cdot\;{\bf w} \:\!)
	\,+\,
	\chi \, \nabla \times\, \mbox{\boldmath $u$}
	\,-\, 2 \, \chi \, {\bf w},
	\end{equation}
	\begin{equation}
	\nabla \cdot \mbox{\boldmath $u$}(\cdot,t)  = 0,
	\end{equation}
\end{subequations}
with initial data 
$ 
\zz_0:=(\:\!\uu_0, {\bf w}_0) \in
{\bf L}^2_{\sigma}(\mathbb{R}^{n})
\!\times\!
{\bf L}^2(\mathbb{R}^{n}) $
and  
\begin{equation*}
\| \zz(\cdot,t) - \zz_0 \|_{L^2(\RR^n)}  \to 0 \mbox{, as } t \to 0,
\end{equation*}
where $\zz:= (\uu,\ww) $. 
In system (\ref{micropolar}), the constants
$ \mu, \gamma,\kappa > 0 $ are kinematic and spin viscosities and $ \chi > 0 $ is the vortex viscosity. The functions 
\mbox{$ \mbox{\boldmath $u$} = \mbox{\boldmath $u$}(\xx,t) $}, \mbox{$ \mbox{\boldmath $w$} = \mbox{\boldmath $w$}(\xx,t) $}, 
and $ \sc{p} = \sc{p}(\xx,t) $
are the flow velocity, micro-rotational velocity and the total pressure, respectively, for $ t>0 $ and $ x \in \RR^n $.
Here,
$ \!\;\!\mbox{\boldmath $L$}^{2}_{\sigma}(\mathbb{R}^{n}) $
denotes the
space
of solenoidal fields
$ \:\!\mbox{\bf v} = (v_{1}, v_{2}, \ldots, v_{n}) \!\:\!\in
\mbox{\boldmath $L$}^{2}(\mathbb{R}^{n}) \equiv L^{2}(\mathbb{R}^{n})^{n} \!\:\!$
with
$ \nabla \!\cdot \mbox{\bf v} \!\;\!= 0 $
in the distributional sense. For a discussion of the physical meaning and the derivation of system \eqref{micropolar}, see \cite{MR167060,MR0204005,MR1711268}. Observe that taking $\ww=0$ in system \eqref{micropolar}, one obtains the classical incompressible Navier-Stokes equations. 

Consider the linear system obtained simply by dropping out the nonlinear terms from system \eqref{micropolar}

\begin{subequations}\label{eqn:linear-system}
	\begin{equation}\label{L_uu_equation}
	\mbox{\boldmath $\bar{u}$}_t
	\;\!+\,
	\nabla \:\!{ \sf{\bar{p}}}
	\;=\;
	(\:\!\mu + \chi\:\!) \, \Delta \mbox{\boldmath $\bar{u}$}
	\,+\,
	\,\chi \, \nabla \times {\bf \bar{w}},
	\end{equation}
	\begin{equation}\label{L_ww_equation}
	{\bf  \bar{w}}_t
	\;\!+\,
	\gamma \;\! \, \Delta {\bf {\bar{\ww}}}
	\,+\,\kappa\, \nabla \:\!(\:\!\nabla \cdot\;{\bf \bar{w}} \:\!)
	\,+ \,\chi \, \nabla \times\, \mbox{\boldmath $\bar{u}$}
	\,-\, 2 \, \chi \, {\bf \bar{w}},
	\end{equation}
	\begin{equation}
	\nabla \cdot \mbox{\boldmath $\bar{u}$}(\cdot,t)  = 0.
	\end{equation}
\end{subequations}
Our main result compares the evolution of solutions $\zz  = ( \mathbf{u} , \mathbf{w})$ for the nonlinear initial value problem \eqref{micropolar} with the evolution 
of solutions $\bar{\zz} = ( \mathbf{\bar{u}} , \mathbf{\bar{w}})$
 for the linear problem \eqref{eqn:linear-system} with the same initial conditions.

%
%

For convenience, define $ \nu:= \min\{\mu,\gamma\} $ and $\displaystyle \lambda_0(\alpha):= \limsup_{t \to \infty} \,t^\alpha\,\|\uu(\cdot,t) \|_{L^2(\RR^3)}$, for $ \alpha \geq 0 $.
We assume that 
\begin{equation}\label{eqn:lambda-assumption}
\lambda_0(\alpha) < \infty,
\end{equation}
$ \text{for some }\,\,\alpha \geq 0$. 
\begin{remark}
In \cite{MR3493117,MR2493562, MR3355116}, the authors define the so called decay indicator. With this tool, a function $ \vv_0 \in L^2 $ is classified accurately with respect to the
algebraic decay rate of the flow initiated from $ \vv_0 $. From there, the authors provide upper bounds for solutions of the system 
in terms of the so called decay character. All of these facts ensure that the assumption \eqref{eqn:lambda-assumption} above holds for some values of $ \alpha \geq 0 $.  
\end{remark}
We state our results. 
\begin{thm}\label{thm1}
	Let $ \zz_0 \in {\bf L}^2_\sigma(\RR^3) \times {\bf L}^2(\RR^3)$. If 
	$\zz(\cdot,t)  = (\mathbf{u} ( \cdot , t) , \mathbf{w}( \cdot , t) )$ 
	is a weak solution of system \eqref{micropolar} with initial data $\zz_0$, then
	\begin{subequations}\label{inequality}
		\begin{equation}\label{inequality_z}
		\limsup_{t \to \infty}\, t^{\alpha+\frac{m}{2}} \norm{D^m \zz(\cdot,t)}_\LT \leq  K_{\alpha,m}\;\;\nu^{-m/2}\;\; \lambda_0(\alpha)
		\end{equation}
		and, for the micro-rotational field $\mathbf{w}$, 
		\begin{equation}\label{inequality_w}
		\limsup_{t \to \infty} t^{\alpha+\frac{m+1}{2}} \norm{D^m \ww(\cdot,t)}_{\LT} \leq  \frac{K_{\alpha,m+1}}{2}\;\;\nu^{-\frac{m+1}{2}} \;\; \lambda_0(\alpha)
		\end{equation}	
	\end{subequations}
	for every $ \alpha \geq0 $ and $ m \geq 1 $, where
	\begin{equation}\label{cte_algebraic}
	K_{\alpha,m} =  \min_{\delta>0} \Bigg\{ \delta^{-1/2} \prod_{j=0}^{m} (\alpha + \frac{j}{2} +\delta)^{1/2}  \Bigg\}. 
	\end{equation}	
\end{thm}	
By the Calderón-Zygmund theory, one directly obtains the following high order inequality for the pressure:
\begin{cor} The pressure $ p ( \cdot , t )$ satisfies 
	\begin{equation*}
	\limsup_{t \to \infty}\, t^{2\alpha+\frac{m}{2}+\frac{3}{4}} \norm{D^m p(\cdot,t)}_\LT \leq  C_m\; \nu^{-(\frac{m}{2}+\frac{3}{4})} \, \lambda_0(\alpha)\,^2 ,
	\end{equation*}
	where $C_m$ is a positive constant depending only on $m$.
\end{cor} 

Our main result gives asymptotic inequalities for high order derivatives of arbitrary weak solutions for systems \eqref{micropolar} and \eqref{eqn:linear-system}, so describing their long time dynamics. To this end, we choose $t_0 > 0$, 
sufficiently large to assure that $\bar{\zz}$ is 
smooth for $t \geq t_0$. Existence of such $t_0$ is assured by Theorem \ref{regularity_thm}. 
Then, consider the solution $\bar{\zz}$ of the linear system \eqref{eqn:linear-system} with initial data $ \bar{\zz}(\cdot,t_0) := \zz(\cdot,t_0) $. We denote such solution by $\bar{\zz} (\cdot , t ; t_0) = ( \bar{\uu} (\cdot , t ; t_0) , \bar{\ww} (\cdot , t ; t_0))$ and, to keep the notation simple, write $\Ez(\cdot,t;t_0) := \zz ( \cdot , t) -  \bar{\zz} (\cdot , t ; t_0)$ and 
$\Ew(\cdot,t;t_0):=\ww(\cdot,t) - \bar{\ww}(\cdot,t;t_0)$. Our main result is the following:
\begin{thm}\label{thm2}
	Let $ \zz_0 \in {\bf L}^2_\sigma(\RR^3) \times {\bf L}^2(\RR^3)$ and $ \zz(\cdot,t) $ be any weak solution to \eqref{micropolar}. If $ 0 \leq\alpha \leq \frac{5}{4} $, then
	\begin{subequations}\label{error_inequality} 
		\begin{equation}\label{error_inequality_z}
		\limsup_{t \to \infty} t^{\alpha + \beta + \frac{m}{2}} \norm{D^m \Ez(\cdot,t;t_0)}_\LT \leq C(\nu,\chi,\alpha,m), \quad 
		\forall t_0 \geq 0,
		\end{equation}
		and, for the micro-rotational field,
		\begin{equation}\label{error_inequality_w}
		\limsup_{t \to \infty} t^{\alpha + \beta + \frac{m+1}{2}} \norm{D^m \Ew(\cdot,t;t_0)}_\LT \leq \frac{1}{2} C(\nu,\chi,\alpha,m+1), \quad \forall t_0 \geq 0,	
		\end{equation}	
		where
		\[
		\beta =
		\begin{cases}
		\alpha + 1/4, & 0 \leq \alpha < 1/2 , \\
		1/4, & 1/2 \leq \alpha < 1  , \\
		5/4 - \alpha, & 1 \leq \alpha < 5/4 ,
		\end{cases}
		\]
		\[
		C(\nu,\chi,\alpha,m) =
		\begin{cases}
		(\nu^{-(\frac{m}{2} + \frac{5}{4})} \,\tilde{K}_{\alpha,m} + c^{-(\frac{m}{2} + \frac{5}{4})}\, 2^{\alpha+\beta+m/2})\,\lambda_0(\alpha)^2,		& 0 \leq \alpha < 1/2, \\
		\nu^{-(\frac{m}{2} + \frac{5}{4})}\,\tilde{K}_{\alpha,m}\,\lambda_0(\alpha)^2 + c^{-(\frac{m}{2} + \frac{5}{4})}\, \norm{\zz_0} 2^{\alpha+\beta+m/2}\,\lambda_0(\alpha),	& 
		1/2 \leq \alpha < 1, \\
		\nu^{-(\frac{m}{2} + \frac{5}{4})}\,\tilde{K}_{\alpha,m}\,\lambda_0(\alpha)^2 + c^{-(\frac{m}{2} + \frac{5}{4})}\, \norm{\zz_0} 2^{\alpha+\beta+m/2},		& 1 \leq \alpha < 5/4,
		\end{cases}
		\]
		with
		\begin{equation*}
		\tilde{K}_{\alpha,m} = \sum_{l=0}^{m} K_{\alpha,l}^{1/4} K_{\alpha,l+1}^{3/4} K_{\alpha,m-l+1}^{1/4} K_{\alpha,m-l+2}^{3/4},
		\end{equation*}	
		\[
		K_{\alpha,m} =  \min_{\delta>0} \Bigg\{ \delta^{-1/2} \prod_{j=0}^{m} (\alpha + \frac{j}{2} +\delta)^{1/2}  \Bigg\}
		\]
		and $c$ is a positive constant given by Lemma \ref{lema-linear_opertaror} below.
	\end{subequations}
\end{thm}
\begin{remark}
Theorem \ref{thm2} actually holds for any initial time $ t_0 \geq 0 $, due to Theorem \ref{thm_linear_aprox} below. 
Moreover, the $ \dot{H}^m$ decay of the related errors depend only on $\limsup_{t \to \infty }\|\uu(t) \|=: \lambda_0(\alpha)$. As 
already mentioned, such inequality has not been observed in the literature even for the Navier-Stokes case. 
\end{remark}
\begin{remark}
It is interesting to note that $\lambda_{0}(\alpha)$ depends only on $\uu$. This is due to the fact that $\ww (\cdot,t)$ goes to zero faster than $\uu (\cdot,t)$ (see \eqref{zz_to_zero} and \eqref{ww_to_zero} below). We should also recall that $\lambda_{0}(0) = 0$ for any weak solution $(\uu,\ww)$ (see \cite{MR3906315} for instance) and that for $(\uu_0, \ww_0) \in L^1(\RR^3)\cap L^2_\sigma(\RR^3)\times L^1(\RR^3)\cap L^2(\RR^3)$ one has $\lambda_{0}(3/4)<\infty$ (see \cite{MR3825173}).
\end{remark}

There are many results on the existence and uniqueness of solutions for problems involving
the system (\ref{micropolar})$ - $ see, e.g., \cite{MR0204005,MR1711268,MR2646523,MR2860636,MR0467030,MR2158216}. For example, in \cite{MR0467030} Galdi and Rionero showed existence and uniqueness of weak solutions for the initial boundary-value problem for the micropolar system (in this case, $\Omega \subset \RR^3$ is a connected
open set that replaces the whole space $\RR^3$ in (\ref{micropolar}) such that the solution vanishes on  $\partial \Omega \times [0,T]$). For the same problem, Lukaszewicz \cite{MR1041744} proved existence and uniqueness of strong solutions in 1989, and, in \cite{MR1711268}, established global existence of weak solutions for arbitrary initial data  $ (\uu_0,\ww_0) \in {{\bf L_\sigma^2 \times L^2}} $. In \cite{MR1484679}, Rojas-Medar proved local existence and uniqueness of strong solutions. In \cite{MR1810322}, 
Ortega-Torres and
Rojas-Medar, assuming small initial data, proved global existence of a
strong solution. 
In \cite{MR2646523}, Boldrini, Durán and Rojas-Medar proved
existence and uniqueness of strong solutions in $L^p(\Omega)$, for $p > 3$. In \cite{MR3516831}, Braz e Silva, Friz and Rojas-Medar showed a exponential decay result for strong solutions in bounded domains on $ \RR^3 $. In \cite{MR3977513}, Cruz showed global existence and uniqueness of strong solutions for sufficiently small initial conditions.

Here, we will extensively use the following properties of the linear and micro-rotational velocities derived in Guterres {\it et al.} \cite{MR3906315}:
\begin{equation}\label{zz_to_zero}
	\lim_{t \to \infty} t^{\frac{s}{2}} \norm{(\uu,\ww)(\cdot,t)}_{\dot{H}^s(\RR^3)} = 0
\end{equation}
and
\begin{equation}\label{ww_to_zero}
	\lim_{t \to \infty} t^{\frac{s+1}{2}} \norm{\ww(\cdot,t)}_{\dot{H}^s(\RR^3)} = 0,
\end{equation}
for each $s \geq 0$.

It is worth to mention that the techniques in the proofs of Theorems \ref{thm1} and \ref{thm2} can be used to obtain similar results for other dissipative systems with non linear terms of the type $\nabla\cdot(\uu\otimes\uu)$. We briefly discuss how to carry on 
the argument to some other models in section \ref{other_systems}. 

\section{Auxiliary Results} 
In order to prove our theorems, we state some basic facts and results. 
Some of them have not been observed before and might even be of independent interest as, for example, Theorem \ref{thm_linear_aprox}. 
\subsection{Regularity time $ t^* $ and monotonicity of high order derivatives}
 Here, we adapt for the asymmetric fluids equations some results 
 on the eventual smoothness of solutions for the Navier-Stokes equations recently obtained in \cite{MR3907942}. 
 
 The classical regularity result for weak solutions due to Leray \cite{MR1555394}  is the following: If $\uu$ is a weak solution
 of
  \begin{eqnarray*}
	\mbox{\boldmath $u$}_t
	+
	\mbox{\boldmath $u$} \cdot \nabla \mbox{\boldmath $u$}
	+
	\nabla \:\!{ \sf{p}}
	& = &
	\mu  \, \Delta \mbox{\boldmath $u$}, \\ 
		\nabla \cdot \mbox{\boldmath $u$}(\cdot,t)   & =  & 0, \\ 
	 	\uu (\cdot , 0 ) & = & \uu_0 \in {\bf L}^2_\sigma(\RR^3) ,
\end{eqnarray*}
  then there exists $t^*>0$ such that $ \uu(\cdot,t) \in C^\infty (\RR^3 \times (t^*,\infty))$  and 
 \begin{equation}\label{eqn:regularity_property}
 \uu(\cdot,t) \in C^0(\,(t^*, \infty), H^m(\RR^n)\,), \,\,\,\forall m\geq0 .
 \end{equation} 
 Moreover, the regularity time $ t^* \geq 0 $ satisfies 
\begin{equation*} 
t^* < K\, \nu^{-5}\,\|\zz_0\|^4_{L^2(\RR^3)} ,\,\, \text{ with} \,\,\,  K \leq \frac{1}{128\pi} <0.000791572 .
\end{equation*}
 Using  the {\it epochs of regularity property} described in \cite{MR943824,MR1555394}, 
 one can adapt the same argument from \cite{MR3907942} to obtain the following result for the asymmetric case.
\begin{thm}\label{regularity_thm}
	If $ \zz(\cdot,t) $ is a weak solution for problem \eqref{micropolar},  then there exists $ t^{**} \geq 0$ such that 
\begin{equation}\label{eqn:regularity_property_improved}
t^* \leq t^{** }\leq 0.000 464 504 284 \,\, \nu^{-5}\,\|\zz_0\|^4_{L^2(\RR^3)} .
\end{equation}  
\end{thm}
This is a slight improvement of the estimate for the regularity time $ t^* $, which is relevant for numerical simulations.    

The following Sobolev type inequality is very useful to our ends.  
\begin{lemma}\label{lema-Sobolev-inequality}
	For all integers $ m \geq 1$, $ 0\leq\ell\leq m-1 $ and any $ f \in H^{m+1}(\RR^3) $, it holds
	\begin{equation}\label{eqn:Sobolev-inequality} 
	 \|D^\ell f \|_{L^\infty(\RR^3)} \| D^{m-\ell} f\|_{L^2(\RR^3)} \leq \|f\|^{1/2}_{L^2(\RR^3)}  \|D f \|^{1/2}_{L^2(\RR^3)} \|D^{m+1}f \|_{L^2(\RR^3)}.
	 \end{equation}      
\end{lemma} 
\begin{proof}
	Through Fourier transformation, one has
	\begin{equation}\label{eqn:sobolev1}
	\|D^\ell f \|_{L^2(\RR^3)} \leq \| f \|^\theta_{L^2(\RR^3)} \|D^k f \|^{1-\theta}_{L^2(\RR^3)}, \,\,\, \theta = \ell/k,
	\end{equation}  
which is actually valid for any dimension $ n \in \mathbb{N} $. Moreover, 
		\begin{equation}\label{eqn:sobolev2}
	\| f \|_{L^\infty(\RR^3)} \leq \| f \|^{1/4}_{L^2(\RR^3)} \|D^2 f \|^{3/4}_{L^2(\RR^3)}.
	\end{equation}  
	Combining the well-known inequalities \eqref{eqn:sobolev1} and \eqref{eqn:sobolev2}, 
	one directly obtains the desired bound \eqref{eqn:Sobolev-inequality}. 
\end{proof}	
 One uses inequality \eqref{eqn:Sobolev-inequality} to easily adapt the proof of a result in \cite{MR3907942} to the case of asymmetric fluids. 
 \begin{thm}
 	Let $ \zz(\cdot,t) $ be a weak solution for problem \eqref{micropolar}.  There exists $ t^{**}_m \geq t^*$ such that $ \|D^m \zz(\cdot,t) \|_{L^2(\RR^3)} $  is monotonically decreasing everywhere in the interval $ (t^{**}_m, \infty) $. Morover, there exists a positive constant $ K_m $, depending only on $m$, such that
 	\begin{equation}\label{eqn:regularity_property_improved-m}
 	t^{** }_m \leq K_m \,\, \nu^{-5}\,\|\zz_0\|^4_{L^2(\RR^3)} .
 	\end{equation}
 \end{thm}     
\subsection{Linear operators leading to asymptotic behavior} We observe that the linear system 
\eqref{eqn:linear-system} has a solution $ \bar{\zz}(\cdot,t;t_0) \in C^0([t_0,\infty), L^2(\RR^3)) $, for each initial time $ t_0 \geq 0 $, that is,  
$\bar{\zz}(\cdot,t;t_0) = e^{\mathbb{A} \:\!t}  \zz(\cdot,t_0) 
(\cdot)$, where 
$$
\mathbb{A}\zz:=\left[
\begin{array}{ccc}
(\mu+\chi)\Delta & \chi \;\!\nabla\wedge  \\
\chi \;\!\nabla\wedge &  \mathbb{L} - 2\chi\,\mbox{Id}_{3\times 3} 
\\
\end{array}
\right]\!\!\;\!
\left[
\begin{array}{ccc}
\uu \\
\ww  
\\
\end{array}
\right]\!\!\;\!,$$
and $ \mathbb{L\ww}:= (\mu+\chi)\Delta \ww+ \kappa \nabla(\nabla\cdot \ww)  $ is the Lamè operator. 
In the frequency space, after taking the Fourier transform of system (\ref{eqn:linear-system}), one obtains

\begin{displaymath}
\partial_t \widehat{\bar{\zz}} = M(\xi) \widehat{\bar{\zz}},
\end{displaymath}
where $M = M(\xi)$ is the matrix of symbols

\begin{equation}
\label{eqn:matrix-symbols}
M = \left( \begin{array}{ccc} - (\mu + \chi) |\xi|^2 Id_{3 \times 3} & i \chi R_3(\xi) 
   \\  i \chi    R_3 (\xi) & - ( \gamma |\xi|^2 + 2 \chi)  Id_{3 \times 3} -  \kappa\,\xi_i \xi_j  
\end{array} \right).
\end{equation}
Here, $Id_{3 \times 3}$ denotes the $3\times 3$ identity matrix and $ i  R_3 (\xi)$ denotes the rotation matrix
\begin{displaymath}
i  R_3 (\xi) =  i  \left( \begin{array}{ccc} 0 & \xi_3 & - \xi_2 \\ - \xi_3 & 0 & \xi_1 \\ \xi_2 & - \xi_1 & 0 \end{array} \right).
\end{displaymath}

In \cite{NichePerusato2020}, the authors proved the following estimate for the eigenvalues of $ M(\xi) $: 
	\begin{displaymath}
	\lambda _{max} (M) \leq - C |\xi|^2, \qquad C =C (\mu, \chi, \gamma) > 0 , 
	\end{displaymath} 
as long as $32 \chi (\mu + \chi + \gamma) > 1$. As a consequence, we immediately have the following upper bound for the semigroup $\bigl(e^{\mathbb{A}t}\bigr)_{t \geq 0}$ associated to \eqref{eqn:linear-system}.     
\begin{lemma}\label{lema-linear_opertaror}
	Let $ \mathcal{G} \in L^2_\sigma(\mathbb{R}^3) \times {L}^2(\mathbb{R}^3) $. If  $32 \chi (\mu + \chi + \gamma) > 1$, then \footnote{Recalling that $ e^{\Delta \tau} $ denotes the heat semigroup.}
	\begin{equation}\label{eqn:linear_operator_eq}
	\|e^{\mathbb{A}t} \:\!\mathcal{G} \|_{L^2} \leq \left\|e^{c\,\Delta t} \:\!\mathcal{G} \right\|_{L^2},
	\end{equation} 
	for all $t \geq 0$. 
\end{lemma} 
It is also worth to note the following result concerning the Lamè semigroup.   
	\begin{lemma}\label{lema-lame-operator}
		If $ \mathcal{G} \in L^2_\sigma(\mathbb{R}^3) \times {L}^2(\mathbb{R}^3) $, then
	\begin{equation}\label{eqn:lame_operator}
\|e^{\mathbb{L}t} \:\!\mathcal{G} \|_{L^2} \leq \left\|e^{c\,\Delta t} \:\!\mathcal{G} \right\|_{L^2},
\end{equation} 
for all $t \geq 0$. 
\end{lemma}
\begin{thm}\label{thm_linear_aprox}
	Let $ \zz(\cdot,t) $ be a weak solution for system \eqref{micropolar} for all $ t>0 $.  If \,   $32 \chi (\mu + \chi + \gamma) > 1$, then, given any pair of initial times $ \tilde{t}_0 \geq t_0 \geq 0  $, one has
	\begin{equation}\label{eqn:Difference_estimate_z} 
		\begin{split}
		\|D^m \bar{\uu}(\cdot,t;{t_0}) - D^m\bar{\uu}(\cdot,t;\tilde{t}_0) \|_{L^2(\RR^3)}  \leq C\, (t-\tilde{t}_0)^{-\frac{5}{4} - \frac{m}{2}},
		\end{split}
	\end{equation} 
		\begin{equation}\label{eqn:Difference_estimate_w}
	\begin{split}
	\|D^m \bar{\ww}(\cdot,t;{t_0}) - D^m\bar{\ww}(\cdot,t;\tilde{t}_0) \|_{L^2(\RR^3)}  \leq C\, e^{-2\,\chi\, (t-\tilde{t}_0)}\,(t-\tilde{t}_0)^{-\frac{5}{4} - \frac{m}{2}}.
	\end{split}
	\end{equation}
	\end{thm} 
\begin{proof} 
 In order to prove \eqref{eqn:Difference_estimate_z} and \eqref{eqn:Difference_estimate_w}, we consider a regularized system as follows (see e.g. \cite{MR673830, MR1555394}):
 Let $ G_\delta $ be a standard mollifier, set $ \tilde{\zz}_{0,\delta}(\cdot) \in C^\infty (\mathbb{R}^6)$ as the convolution of $ \zz_0 (\cdot)$ with $ G_\delta $, $ \delta > 0 $ and define $ \zz_\delta $, $ p_\delta $ as the (unique, globally defined) classical $ L^2 $ solution for the system  
	\begin{equation*}
	\left\{\begin{array}{l}
	\frac{\partial{\uu_\delta}}{\partial t} + (\tilde{\uu}_\delta \cdot \nabla)\uu_\delta -(\mu+\chi) \Delta \uu_\delta + \nabla p_\delta  =   \chi\,\nabla \wedge
	\ww_\delta, \\ \noalign{\medskip}
	\frac{\partial{\ww_\delta}}{\partial t}+ (\tilde{\uu}_\delta \cdot \nabla )
	\ww_\delta -  \gamma\Delta \ww_\delta -
	\kappa\nabla( \nabla \cdot \ww_\delta)
	+ 2 \chi\ww_\delta = \,
	\chi\,\nabla\wedge \uu_\delta, \\ \noalign{\medskip}
	\nabla \cdot \uu_\delta  =0, \\ \noalign{\smallskip}
	\zz_\delta(\cdot,0)= \tilde{\zz}_{0,\delta}:=G_\delta \ast \zz_0 \in \bigcap_{m=1}^\infty {\bf H}^m (\mathbb{R}^6)  ,
	\end{array}\right.
	\end{equation*}
	where $ \tilde{\zz}_\delta (\cdot,t):=G_\delta \ast \zz_\delta (\cdot,t)  $. It was shown by Leray that, for some suitable  sequence $ \tilde{\delta} \to 0 $, one has the $ L^2 $ weak convergence property $ \zz_{\tilde{\delta}} (\cdot,t) \rightharpoonup \zz(\cdot,t) $ as $ \tilde{\delta} \to 0 $, for all $ t \geq 0 $.  Moreover, it is easy to obtain  the energy inequality
	\begin{equation}\label{eqn:weak-inequality}
	\|\zz_\delta (\cdot,t) \|^2_{L^2} \leq 2\nu \int_{0}^{t} \|D \zz_\delta (\cdot,\tau) \|^2_{L^2} d\tau \leq \|\zz_0 \|^2_{L^2}.
	\end{equation} 
	Let $$ \QQ_\delta  (\cdot,t): = \left[
	\begin{array}{ccc}
		(\tilde{\uu}_\delta \cdot \nabla)\uu_\delta - \nabla p_\delta\\
		(\tilde{\uu}_\delta \cdot \nabla)\ww_\delta  
		\\
	\end{array}
	\right]\!\!\;\!$$
and write $ \bar{\zz}(\cdot,t;t_0) $ as 
\begin{equation*}
\bar{\zz}(\cdot,t;t_0) = e^{\mathbb{A}(t-t_0)  } [\zz (\cdot,t_0) - \zz_\delta (\cdot,t_0) ] + e^{\mathbb{A}(t-t_0)} \zz_\delta(\cdot,t_0), \quad t>t_0.
\end{equation*} 
Observing that
\begin{equation*}
\begin{split}
\zz_\delta(\cdot,t) = e^{\mathbb{A}t_0 } \tilde{\zz}_{0,\delta} + \int_{0}^{t_0} e^{\mathbb{A}(t_0 - \tau) } \QQ_\delta(\cdot,\tau) d\tau, 
\end{split}
\end{equation*} 	
we obtain
\begin{equation*}
\bar{\zz}(\cdot,t;t_0) = e^{\mathbb{A}(t-t_0)} \, [\,\zz (\cdot,t_0) - \zz_\delta (\cdot,t_0)  \,] + e^{\mathbb{A} t} \tilde{\zz}_{0,\delta} + \int_{0}^{t_0} e^{\mathbb{A}(t - \tau) } \QQ_\delta(\cdot,\tau) d\tau, \quad \forall t>t_0\geq 0 .
\end{equation*}
Analogously,
\begin{equation*}
\bar{\zz}(\cdot,t;\tilde{t}_0) = e^{\mathbb{A}(t-\tilde{t}_0)} \, [\zz (\cdot,\tilde{t}_0) - \zz_\delta (\cdot,\tilde{t}_0) ] + e^{\mathbb{A} t} \tilde{\zz}_{0,\delta} + \int_{0}^{\tilde{t}_0} e^{\mathbb{A}(t - \tau) } \QQ_\delta(\cdot,\tau) d\tau, \quad \forall t>\tilde{t}_0 \geq t_0 .
\end{equation*}
Therefore, given an arbitrary $\mathbb{K} \subset \RR^3$ compact, it follows, from Lemma \ref{lema-linear_opertaror}, 
\begin{equation}\label{eqn: before-estimate}
\begin{split}
\|D^m\bar{\zz}(\cdot,t;\tilde{t}_0) - D^m\bar{\zz}(\cdot,t;t_0) \|_{L^2(\mathbb{K})} \leq I_{m,\delta}(t) + \int_{t_0}^{\tilde{t}_0} \| e^{\mathbb{A}(t -\tau)} D^m\QQ_\delta(\cdot,\tau)\|_{L^2(\RR^3)} \,d\tau \\
\leq I_{m,\delta}(t) + \int_{t_0}^{\tilde{t}_0} \| e^{\Delta\,(t -\tau)} D^m \QQ_\delta(\cdot,\tau)\|_{L^2(\RR^3)} \,d\tau,
\end{split}
\end{equation}  
where 
\begin{equation*} 
\begin{split} 
 I_{m,\delta}(t):= \|D^m \{e^{\mathbb{A}(t-t_0)} \, [\zz (\cdot,t_0) - \zz_\delta (\cdot,t_0) ]\}\|_{L^2(\mathbb{K})} \\+ \|D^m \{ e^{\mathbb{A}(t-\tilde{t}_0)} \, [\zz (\cdot,\tilde{t}_0) - \zz_\delta (\cdot,\tilde{t}_0) ]\} \|_{L^2(\mathbb{K})}  .
 \end{split} 
 \end{equation*}
 Since $ \|D^m e^{\Delta (t - \tau)}  {\bf Q}_\delta(\cdot,\tau) \|_{L^2} \leq C (t-\tau)^{-\frac{5}{4} - \frac{m}{2}}  \|\zz_\delta (\cdot,\tau) \|^2_{L^2} $, one has
 \begin{equation}\label{eqn: Q-estimate}
 \begin{split}
 \int_{t_0}^{\tilde{t}_0} \| e^{\Delta\,(t -\tau)} D^m \QQ_\delta(\cdot,\tau)\|_{L^2(\RR^3)} \,d\tau \leq C\,\int_{t_0}^{\tilde{t}_0}(t-\tau)^{-\frac{5}{4} - \frac{m}{2}} \| \zz_\delta (\cdot,\tau)\|^2_{L^2(\RR^3)} \,d\tau \\
 \leq 
 C\,\int_{t_0}^{\tilde{t}_0}(t-\tau)^{-\frac{5}{4} - \frac{m}{2}} \| \zz_0\|^2_{L^2(\RR^3)} \,d\tau \leq C\, (t-\tilde{t}_0)^{-\frac{5}{4} - \frac{m}{2}} \| \zz_0\|^2_{L^2(\RR^3)} (\tilde{t}_0- t_0)\\ = C\,(t-\tilde{t}_0)^{-\frac{5}{4} - \frac{m}{2}}, 
 \end{split}
 \end{equation}
 for some $C > 0$, where we used bound \eqref{eqn:weak-inequality}. 
 
 Now,  taking $ \delta = \tilde{\delta} \to 0  $, by the $ L^2 $ weak convergence property mentioned above, we get $ I_{m,\delta}(t) \to 0 	 $, since,  for any $ \sigma, s > 0 $, 
 \begin{equation}\label{eqn:L2-convergence}
 \|D^m \,\{\,e^{\Delta s} [\,\zz(\cdot,\sigma) - \zz_\delta (\cdot,\sigma)\,]\,\}\|_{L^2(\mathbb{K})}  \to 0 \quad \text{as} \quad \tilde{\delta} \to 0.
 \end{equation}
 Indeed, defining $F_\delta (\cdot,s) :=D^m \,\{\,e^{\Delta s} [\,\zz(\cdot,\sigma) - \zz_\delta (\cdot,\sigma)\,]\,\} $ one has 
 \begin{equation*}
  F_\delta (\cdot,s)= H_m(\cdot,s) \ast [\zz(\cdot,\sigma) - \zz_\delta(\cdot,\sigma) ] , 
 \end{equation*}
 where $ H_m \in L^1(\RR^3) \cap L^\infty(\RR^3) $. Since $ \zz(\cdot,\sigma) - \zz_\delta(\cdot,\sigma)  \rightharpoonup 0 $ in $ L^2 $, it follows that 	$ F_{\tilde{\delta}} \to 0 $ for each $ x \in \RR^3 $. On the other hand, by inequality \eqref{eqn:weak-inequality} and the Cauchy-Schwarz inequality, one obtains
 \begin{equation*}
 \begin{split} 
 |F_\delta(x,s) | \leq \|H_\delta (\cdot,s) \|_{L^2(\RR^3)} \|\zz(\cdot,\sigma) - \zz_\delta(\cdot,\sigma) \|_{L^2(\RR^3)} \\ \leq 2 \|H_\delta (\cdot,s) \|_{L^2(\RR^3)} \|\zz_0 \|_{L^2(\RR^3)}
 \end{split}
 \end{equation*}  
 for all $ x \in \RR^3 $. The Lebesgue`s dominated convergence theorem implies that $ \|F_{\tilde{\delta}}(\cdot,s)  \|_{L^2(\mathbb{K})}  \to 0 $ as $\tilde{\delta} \to 0$. Hence, recalling that $ \mathbb{K} $ is compact, the proof of claim \eqref{eqn:L2-convergence} is complete. 
 
 Letting $ \tilde{\delta} \to 0  $ in \eqref{eqn: before-estimate}, \eqref{eqn:L2-convergence} and \eqref{eqn: Q-estimate}, we obtain
 \begin{equation*}
\|D^m\bar{\zz}(\cdot,t;\tilde{t}_0) - D^m\bar{\zz}(\cdot,t;t_0) \|_{L^2(\mathbb{K})} \leq C (t-\tilde{t}_0)^{-\frac{5}{4} - \frac{m}{2}},
\end{equation*}    
for each $ t>\tilde{t}_0 $ and for {\it any} compact $ \mathbb{K} \subset \RR^3 $. This is equivalent to the desired inequality \eqref{eqn:Difference_estimate_z}. 

To show inequality \eqref{eqn:Difference_estimate_w}, we work with the linearized micro-rotational equation \eqref{L_ww_equation} separately and take advantage of the linear damping term $ 2\,\chi \bar{\ww} $. To this end, just write $ \bar{\WW} : = e^{2\chi (t-t_0)} \bar{\ww} $, apply Duhamel's principle for $\bar{\WW}$ and write the result in terms of $ \bar{\ww}(\cdot,t;t_0) $ to get
\begin{equation*}
\begin{split}
\bar{\ww}(\cdot,t;t_0) =  e^{-2\,\chi (t-t_0) }\,e^{\mathbb{L}(t-t_0) } [\ww(\cdot,t_0) - \ww_\delta (\cdot,t_0) ] + \chi \int_{t_0}^{t} e^{-2 \chi (t-s)} e^{\mathbb{L} (t-s) } \nabla \wedge \bar{\uu}(\cdot,s) ds \\
+ e^{-2\chi (t-t_0)}e^{\mathbb{L} (t-t_0) } \ww_\delta(\cdot,t_0). 
\end{split}
\end{equation*}
Furthermore, from system \eqref{eqn:linear-system}, it also holds
\begin{equation*}
\begin{split}
e^{-2\chi (t-t_0)} e^{\mathbb{L}(t-t_0)} \ww_\delta (\cdot,t_0) = e^{-2\chi\,t} e^{\mathbb{L}\,t} \bar{\ww}_{0,\delta} + \chi\,\int_{0}^{t_0} e^{-2\chi (t-s)} e^{\mathbb{L}(t-s)} \nabla \wedge \bar{\uu}_{\delta} (\cdot,s)  ds \\
- \int_{0}^{t_0} e^{-2\chi (t-s) } e^{\mathbb{L}(t-s)}[(\tilde{\uu}_\delta \cdot \nabla)\ww_\delta ](\cdot,s) ds. 
\end{split}
\end{equation*}  
Consequently, after adding $ \pm \int_{0}^{t_0} e^{-2\,\chi (t-s)} e^{\mathbb{L}\,(t-s)} \nabla \wedge \bar{\uu} (\cdot,t)ds $, we get
\begin{equation*}
\begin{split}
\bar{\ww}(\cdot,t;t_0) =  e^{-2\,\chi (t-t_0) }\,e^{\mathbb{L}(t-t_0) } [\ww(\cdot,t_0) - \ww_\delta (\cdot,t_0) ] + e^{-2\chi \,t} e^{\mathbb{L}\,t } \bar{\ww}_{0,\delta} \\+ \chi \int_{t_0}^{t} e^{-2 \chi (t-s)} e^{\mathbb{L} (t-s) } \nabla \wedge \{\bar{\uu} - \bar{\uu}_\delta \}(\cdot,s)  ds + \chi \int_{0}^{t} e^{-2 \chi (t-s)} e^{\mathbb{L} (t-s) } \nabla \wedge \bar{\uu}(\cdot,s) ds \\
- \int_{0}^{t_0} e^{-2\chi (t-s) } e^{\mathbb{L}(t-s)}[(\tilde{\uu}_\delta \cdot \nabla)\ww_\delta ](\cdot,s) ds, \quad \forall t > t_0,
\end{split}
\end{equation*}
 for each $ t_0 \geq 0$. Therefore, by Lemma \ref{lema-lame-operator}
 \begin{equation}\label{eqn:before_estimate_w}
 \begin{split}
 \|D^m\bar{\ww}(\cdot,t;t_0)  - D^m\bar{\ww}(\cdot,t;\tilde{t}_0) \|_{L^2(\mathbb{K})} \leq \mathcal{J}_{m,\delta} (t) \\+ \int_{t_0}^{\tilde{t}_0} e^{-2\chi (t-s) } \| D^m e^{\mathbb{L}(t-s)}(\tilde{\uu}_\delta \cdot \nabla)\ww_\delta (\cdot,s) \|_{L^2(\mathbb{K})} ds \\ \leq \mathcal{J}_{m,\delta} (t) + \int_{t_0}^{\tilde{t}_0} e^{-2\chi (t-s) } \| D^m e^{\Delta(t-s)}(\tilde{\uu}_\delta \cdot \nabla)\ww_\delta (\cdot,s) \|_{L^2(\mathbb{K})} ds,
 \end{split}
 \end{equation}
where 
\begin{equation*}
\begin{split}
\mathcal{J}_{m,\delta} (t): =  e^{-2\,\chi (t-t_0) }\|D^m\{\,e^{\Delta(t-t_0) } [\ww(\cdot,t_0) - \ww_\delta (\cdot,t_0) ] \}\|_{L^2(\mathbb{K})}   \\ +\,\,
 e^{-2\,\chi (t-\tilde{t}_0) }\|D^m\{\,e^{\Delta(t-\tilde{t}_0) } [\ww(\cdot,\tilde{t}_0) - \ww_\delta (\cdot,\tilde{t}_0) ] \}\|_{L^2(\mathbb{K})} \\+\,\,
  \chi \int_{t_0}^{t} e^{-2 \chi (t-s)} \|D^{m+1}e^{\Delta (t-s) }\{  \bar{\uu} - \bar{\uu}_\delta \}(\cdot,s) \|_{L^2(\RR^3)} ds,
\end{split}
\end{equation*} 
for an arbitrary compact $ \mathbb{K} \subset \RR^3 $. As before, use the energy inequality \eqref{eqn:weak-inequality} to obtain
\begin{equation}\label{eqn: Q-estimate-w}
\begin{split}
\int_{t_0}^{\tilde{t}_0} e^{-2\chi (t-s) } \| D^m e^{\Delta(t-s)}(\tilde{\uu}_\delta \cdot \nabla)\ww_\delta (\cdot,s) \|_{L^2(\RR^3)} ds \\ \leq e^{-2\chi (t-\tilde{t}_0) }\int_{t_0}^{\tilde{t}_0}(t-s)^{-\frac{5}{4} - \frac{m}{2}}  \| \zz_\delta (\cdot,s) \|^2_{L^2(\RR^3)} ds  \leq C\,e^{-2\chi (t-\tilde{t}_0) }(t-\tilde{t}_0)^{-\frac{5}{4} - \frac{m}{2}}, \,\,\, [\text{by} \eqref{eqn:weak-inequality}]
\end{split}
\end{equation}
with $ C = C(t_0,\tilde{t}_0, \|\zz_0\|) > 0 $. By the same argument as before (c.f. \eqref{eqn:L2-convergence}), one has
 \begin{equation}\label{eqn:L2-convergence-w}
 \mathcal{J}_{ m,\delta}(t) \to 0 \quad \text{as} \quad \delta = \tilde{\delta} \to 0.
 \end{equation}
 Therefore, by bounds \eqref{eqn:before_estimate_w}, \eqref{eqn: Q-estimate-w} and 
 the convergence \eqref{eqn:L2-convergence-w}, it follows that   
 \begin{equation*}
  \|D^m\bar{\ww}(\cdot,t;t_0)  - D^m\bar{\ww}(\cdot,t;\tilde{t}_0) \|_{L^2(\mathbb{K})} \leq  C\,e^{-2\chi (t-\tilde{t}_0) }(t-\tilde{t}_0)^{-\frac{5}{4} - \frac{m}{2}},
 \end{equation*}
 for each $ t>\tilde{t}_0 $ and for {\it any} compact $ \mathbb{K} \subset \RR^3 $. This is clearly equivalent to the desired inequality \eqref{eqn:Difference_estimate_w},
\end{proof}

\section{Proof of  Theorem \ref{thm1}}
Define $ \tilde{\lambda}_0(\alpha):= \limsup_{t \to \infty} t^\alpha \|\zz(\cdot,t) \|$. 
We first assume $\tilde{\lambda}_0(\alpha) < \infty$ and prove
inequalities \eqref{inequality} with $\tilde{\lambda}_0 (\alpha )$ replacing ${\lambda}_0(\alpha) = \limsup_{t \to \infty} t^\alpha \|\uu(\cdot,t) \|$ in their right hand side. Finally, at the end of the section 
(see Claim \ref{rmk-remove-w} below), we finish the proof showing that $ \tilde{\lambda}_0(\alpha) = \lambda_0(\alpha)$ whenever $\lambda_0(\alpha)< \infty$.
\subsection{Proof of bound \eqref{inequality_z}}
For simplicity, we define the scalar function $ \zeta (\cdot,t):= \nabla \cdot \ww (\cdot,t) $ and $ \nu:=\min\{\mu,\gamma\}  $ as before.
Let $t \geq t_0 > t_*$, where $t_*$ is the regularity time (see \eqref{eqn:regularity_property}).	Given $\delta > 0$, taking the dot product of equations (\ref{uu_equation}) and (\ref{ww_equation}) with $2(s-t_0)^{2\alpha + \delta} \uu$ and $2(s-t_0)^{2\alpha + \delta} \ww$ respectively, and integrating by parts over $[t_0,t]\times\RR^3$, we get
\begin{equation}\label{sup_e1}
\begin{split}
(t-t_0)^{2\alpha+\delta}\norm{\zz(\cdot,t)}_\LL^2 + 2\nu\int_{t_0}^{t}(s-t_0)^{2\alpha+\delta}\norm{D\zz(\cdot,s)}_\LL^2 ds\\ +2\kappa \int_{t_0}^{t}(s-t_0)^{2\alpha+\delta}\norm{\zeta (\cdot,s)}_\LL^2 ds +2\chi \int_{t_0}^{t}(s-t_0)^{2\alpha+\delta}\norm{\ww (\cdot,s)}_\LL^2 ds  \\ \leq (2\alpha + \delta)\int_{t_0}^{t}(s-t_0)^{2\alpha+\delta-1}\norm{\zz(\cdot,s)}_\LL^2 ds.
\end{split}
\end{equation}
Due to assumption \eqref{eqn:lambda-assumption}, given  $0 < \epsilon < 2$, we can choose $t_0>t_*$ sufficiently large such that $t^\alpha \norm{\zz(\cdot,t)}_\LT < \tilde{\lambda}_0(\alpha) + \epsilon$. Combining this with bound (\ref{sup_e1}), we get
\begin{equation}\label{sup_e2}
\begin{split}
(t-t_0)^{2\alpha+\delta}\norm{\zz(\cdot,t)}_\LL^2 + 2\nu\int_{t_0}^{t}(s-t_0)^{2\alpha+\delta}\norm{D\zz(\cdot,s)}_\LL^2 ds\\ +2\kappa \int_{t_0}^{t}(s-t_0)^{2\alpha+\delta}\norm{\zeta (\cdot,s)}_\LL^2 ds +2\chi \int_{t_0}^{t}(s-t_0)^{2\alpha+\delta}\norm{\ww (\cdot,s)}_\LL^2 ds  \\ \leq (2\alpha + \delta)(\tilde{\lambda}_0(\alpha)+ \epsilon)^2\int_{t_0}^{t}(s-t_0)^{2\alpha+\delta-1}s^{-2\alpha} ds \\
\leq (2\alpha + \delta)(\tilde{\lambda}_0(\alpha)+ \epsilon)^2 \frac{(t-t_0)^\delta}{\delta}.
\end{split}
\end{equation}	 
In particular,
\begin{equation*}
(t-t_0)^{2\alpha}\norm{\zz(\cdot,t)}_\LL^2 \leq \frac{2\alpha + \delta}{\delta}(\tilde{\lambda}_0(\alpha)+ \epsilon)^2 
\end{equation*}
and
\begin{equation}\label{sup_e3}
\int_{t_0}^{t}(s-t_0)^{2\alpha+\delta}\norm{D\zz(\cdot,s)}_\LL^2 ds \leq \frac{1}{2\nu}(2\alpha + \delta)(\tilde{\lambda}_0(\alpha)+ \epsilon)^2 \frac{(t-t_0)^\delta}{\delta},
\end{equation}
for all $t \geq t_0$.	
For the next step,  first observe that  differentiating equations (\ref{uu_equation}) and (\ref{ww_equation}) with respect to $x_l$ (for example), taking the dot product of equations (\ref{uu_equation}) and (\ref{ww_equation}) by $2(s-t_0)^{2\alpha + \delta+1}D_l \uu$ and $2(s-t_0)^{2\alpha + \delta+1}D_l \ww$, respectively, and integrating the result over $\RR^3 \times [t_0,t]$, we get, after using Lemma \ref{lema-Sobolev-inequality} and summing over $1 \leq l \leq 3$,
\begin{equation*}
\begin{split}
(t-t_0)^{2\alpha+\delta+1}\norm{D\zz(\cdot,t)}_\LL^2\\ + 2\nu\int_{t_0}^{t}(s-t_0)^{2\alpha+\delta+1}\norm{D^2 \zz(\cdot,s)}_\LL^2 ds\\ +2 \kappa \int_{t_0}^{t}(s-t_0)^{2\alpha+\delta+1}\norm{D\zeta (\cdot,s)}_\LL^2 ds +2\chi \int_{t_0}^{t}(s-t_0)^{2\alpha+\delta+1}\norm{D\ww (\cdot,s)}_\LL^2 ds  \\ \leq (2\alpha + \delta+1)\int_{t_0}^{t}(s-t_0)^{2\alpha+\delta}\norm{D\zz(\cdot,s)}_\LL^2 ds \\ + C\int_{t_0}^{t}(s-t_0)^{2\alpha+\delta+1} \norm{\zz(\cdot,s)}_{L^\infty(\RR^3)} \norm{D\zz(\cdot,s)}_\LL \norm{D^2\zz(\cdot,s)}_\LL ds \\
\leq 
 (2\alpha + \delta+1)\int_{t_0}^{t}(s-t_0)^{2\alpha+\delta}\norm{D\zz(\cdot,s)}_\LL^2 ds \\ + C\int_{t_0}^{t}(s-t_0)^{2\alpha+\delta+1} \norm{\zz(\cdot,s)}_{\LL}^{1/2} \norm{D\zz(\cdot,s)}_\LL^{1/2}\norm{D^2\zz(\cdot,s)}_\LL^2 ds.
\end{split}
\end{equation*}
Since $ \|D\zz(\cdot,t) \|_{L^2(\RR^3)} \to 0  $ as $ t $ tends to infinity\footnote{Actually, for arbitrary initial data $\zz_0 \in L^2  $, a faster decay rate holds for the gradient: $ t^{1/2} \|D\zz \| \to 0 $ as $ t \to \infty$. See e.g \cite{MR3854334}, for more details.},  increasing $ t_0 $ if necessary, we obtain 
\begin{equation*}
\begin{split}
(t-t_0)^{2\alpha+\delta+1}\norm{D\zz(\cdot,t)}_\LL^2\\ + (2-\epsilon)\,\nu\int_{t_0}^{t}(s-t_0)^{2\alpha+\delta+1}\norm{D^2 \zz(\cdot,s)}_\LL^2 ds\\ +2 \kappa \int_{t_0}^{t}(s-t_0)^{2\alpha+\delta+1}\norm{D\zeta (\cdot,s)}_\LL^2 ds +2\chi \int_{t_0}^{t}(s-t_0)^{2\alpha+\delta+1}\norm{D\ww (\cdot,s)}_\LL^2 ds  \\ \leq (2\alpha + \delta+1)\int_{t_0}^{t}(s-t_0)^{2\alpha+\delta}\norm{D\zz(\cdot,s)}_\LL^2 ds.
\end{split}
\end{equation*}	   
This, with inequality \eqref{sup_e3}),  immediately gives
\begin{equation*}
(t-t_0)^{2\alpha+1}\norm{D\zz(\cdot,t)}_\LL^2 \leq \frac{1}{(2-\epsilon)\delta\nu}(2\alpha + \delta+1)(2\alpha + \delta)(\tilde{\lambda}_0(\alpha)+ \epsilon)^2 
\end{equation*}
and
\begin{equation*}
\begin{split}
\int_{t_0}^{t}(s-t_0)^{2\alpha+\delta+1}\norm{D^2\zz(\cdot,s)}_\LL^2 ds\\ \leq \frac{1}{\delta((2-\epsilon)\nu)^2}(2\alpha + \delta+1)(2\alpha + \delta)(\tilde{\lambda}_0(\alpha)+ \epsilon)^2 (t-t_0)^\delta ,
\end{split}
\end{equation*}	 
which proves \eqref{inequality_z}, when $ \tilde{\lambda}_0(\alpha) < \infty$, for $ m=1$. Now, we proceed by an induction argument. Assume that, for some $ m \geq 2 $,
\begin{subequations}
	\begin{equation}
	(t-t_0)^{2\alpha + k} \norm{D^k \zz (\cdot,t)}^2_\LT \leq \frac{1}{\delta [(2-\epsilon)\lambda]^{k}} \bigg{[} \prod_{j=0}^{k} (2\alpha + j + \delta) \bigg{]} (\tilde{\lambda}_0(\alpha) + \epsilon)^2
	\end{equation}
	and
	\begin{equation}\label{des_ind2}
	\begin{split} 
	\int_{t_0}^{t} (s-t_0)^{2\alpha + k + \delta} \norm{D^{k+1} \zz (\cdot,s)}^2_\LT ds \\ \leq \bigg{(} \frac{1}{\delta [(2-\epsilon)\lambda]^{k+1}} \bigg{[} \prod_{j=0}^{k} (2\alpha + j + \delta) \bigg{]} (\tilde{\lambda}_0(\alpha) + \epsilon)^2 \bigg{)}(t-t_0)^\delta
	\end{split} 
	\end{equation}
\end{subequations}
hold for every $1 \leq k \leq m-1$, and for all $t \geq t_0$ (for any $t_0 > t_{**}$, for some $t_{**}>t_*$ sufficiently large, see \eqref{eqn:regularity_property}, \eqref{eqn:regularity_property_improved}). Then,  
applying $D_{l_1} D_{l_2} ... D_{l_m}$ to the equation (\ref{uu_equation}) and taking the dot product with $2(s-t_0)^{2\alpha+m+\delta}D_{l_1} D_{l_2} ... D_{l_m} \uu(\cdot,s)$ we get, (integrating over $\RR^3 \times [t_0,t]$ and summing over $1 \leq l_1, l_2, ... , l_m \leq 3$):

\begin{equation*}
\begin{split}
(t-t_0)^{2\alpha + m + \delta} \norm{D^m \uu(\cdot,t)}^2_\LT + 2 (\mu + \chi) \int_{t_0}^{t} (s-t_0)^{2\alpha + m + \delta} \norm{D^{m+1} \uu(\cdot,s)}^2_\LT ds \\
= (2\alpha + m + \delta) \int_{t_0}^{t} (s-t_0)^{2\alpha + m -1 + \delta} \norm{D^m \uu(\cdot,s)}^2_\LT ds \\ + 2 \sum_{i,j,l_1,l_2,...,l_m=1}^{3} \int_{t_0}^{t} (s-t_0)^{2\alpha + m + \delta} \int_{\RR^3} [D_j D_{l_1}...D_{l_m} \uu_i] [D_{l_1} ... D_{l_m} (\uu_i \uu_j)] dx ds \\
+ 2 \chi \sum_{l_1,...,l_m =1}^{3} \int_{t_0}^{t} (s-t_0)^{2\alpha + m + \delta} \int_{\RR^3} <D_{l_1} ... D_{l_m} \uu,  \nabla \wedge D_{l_1} ... D_{l_m} \ww> dx ds.
\end{split}
\end{equation*}	
Do the same for equation (\ref{ww_equation}). Adding up the resulting inequalities, one has
\begin{equation*}
\begin{split}
(t-t_0)^{2\alpha + m + \delta} \norm{D^m \zz(\cdot,t)}^2_\LT + 2 \nu \int_{t_0}^{t} (s-t_0)^{2\alpha + m + \delta} \norm{D^{m+1} \zz(\cdot,s)}^2_\LT ds \\ + 2 \chi \int_{t_0}^{t} (s-t_0)^{2\alpha + m + \delta} \norm{D^{m+1} \uu	(\cdot,s)}^2_\LT ds \\+ 4 \chi \int_{t_0}^{t} (s-t_0)^{2\alpha + m + \delta} \norm{D^{m} \ww(\cdot,s)}^2_\LT ds \\+ 2 \kappa  \int_{t_0}^{t} (s-t_0)^{2\alpha + m + \delta} \norm{D^{m} \zeta (\cdot,s)}^2_\LT ds \\ \leq (2\alpha + m + \delta) \int_{t_0}^{t} (s-t_0)^{2\alpha + m -1 + \delta} \norm{D^m \zz(\cdot,s)}^2_\LT \\
+K_m \sum_{l=0}^{\big{[} \frac{m}{2} \big{]}} \int_{t_0}^{t} (s-t_0)^{2\alpha + m + \delta}  \norm{D^l \zz(s)}_{L^\infty(\RR^3)} \norm{D^{m-l} \zz(s)}_\LT \norm{D^{m+1} \zz(s)}_\LT ds \\ + 4 \chi \sum_{l_1,...,l_m =1}^{3} \int_{t_0}^{t} (s-t_0)^{2\alpha + m + \delta} \int_{\RR^3} <D_{l_1} ... D_{l_m} \ww(x,s),  \nabla \wedge D_{l_1} ... D_{l_m} \uu(x,s)> dx ds,
\end{split}
\end{equation*}
	for some constant $K_m >0$ (which depends only on $m$). Observing that
\begin{equation*}
\begin{split}
4 \chi \sum_{l_1,...,l_m =1}^{3} (s-t_0)^{2\alpha + m + \delta} \int_{\RR^3} <D_{l_1} ... D_{l_m} \ww,  \nabla \wedge D_{l_1} ... D_{l_m} \uu> dx ds \\
\leq 2\chi \norm{D^m \ww(\cdot,s)}^2_\LT + 2 \chi \norm{D^{m+1} \uu(\cdot,s)}^2_\LT,
\end{split}
\end{equation*}
by Young`s inequality, the bound
\begin{equation*}
\begin{split}
K_m \sum_{l=0}^{\big{[} \frac{m}{2} \big{]}} \int_{t_0}^{t} (s-t_0)^{2\alpha + m + \delta}  \norm{D^l \zz(s)}_{L^\infty(\RR^3)} \norm{D^{m-l} \zz(s)}_\LT \norm{D^{m+1} \zz(s)}_\LT ds \\
\leq K_m \sum_{l=1}^{m-1} \int_{t_0}^{t} (s-t_0)^{2\alpha + m + \delta} \norm{\zz(s)}^{1/2}_\LT \norm{D \zz (s)}^{1/2}_\LT  \norm{D^{m+1} \zz (\cdot,s)}^2_\LT ds,
\end{split}
\end{equation*}
 and Lemma \ref{lema-Sobolev-inequality}, one obtains
\begin{equation}\label{sup_e6}
\begin{split}
(t-t_0)^{2\alpha + m + \delta} \norm{D^m \zz(\cdot,t)}^2_\LT + 2 \nu \int_{t_0}^{t} (s-t_0)^{2\alpha + m + \delta} \norm{D^{m+1} \zz(\cdot,s)}^2_\LT ds  \\+ 2 \chi \int_{t_0}^{t} (s-t_0)^{2\alpha + m + \delta} \norm{D^{m} \ww(\cdot,s)}^2_\LT ds \\+ 2 \kappa  \int_{t_0}^{t} (s-t_0)^{2\alpha + m + \delta} \norm{D^{m} \zeta (\cdot,s)}^2_\LT ds \\ \leq (2\alpha + m + \delta) \int_{t_0}^{t} (s-t_0)^{2\alpha + m -1 + \delta} \norm{D^m \zz(\cdot,s)}^2_\LT \\
+K_m \sum_{l=0}^{\big{[} \frac{m}{2} \big{]}} \int_{t_0}^{t} (s-t_0)^{2\alpha + m + \delta}  \norm{\zz(s)}^{1/2}_{L^2(\RR^3)} \norm{D \zz(s)}^{1/2}_\LT \norm{D^{m+1} \zz(s)}^2_\LT ds . 
\end{split}
\end{equation}
Again, since $\norm{\zz(\cdot,t)}_\LT$ is bounded in $[t_*,\infty)$ and  $\norm{D \zz(\cdot,t)}_\LT \to 0$ as $t \to \infty$ (see e.g \cite{MR3955606, MR3854334}), increasing $t_0 > t_*$ if necessary, inequality (\ref{sup_e6}) becomes\footnote{Although it is not required here, we actually have that $ \|\zz(t) \|$ is bounded in $ [0,\infty) $ and tends to $ 0 $ as $ t \to \infty $. Moreover, as already mentioned, $ t^{1/2}\|D\zz \| \to 0 $.}
\begin{equation}\label{sup_e7}
\begin{split}
(t-t_0)^{2\alpha + m + \delta} \norm{D^m \zz(\cdot,t)}^2_\LT +  (2-\epsilon)\nu \int_{t_0}^{t} (s-t_0)^{2\alpha + m + \delta} \norm{D^{m+1} \zz(\cdot,s)}^2_\LT ds \\ + 2 \chi \int_{t_0}^{t} (s-t_0)^{2\alpha + m + \delta} \norm{D^{m} \ww(\cdot,s)}^2_\LT ds \\+ 2 \kappa \int_{t_0}^{t} (s-t_0)^{2\alpha + m + \delta} \norm{D^{m} \zeta (\cdot,s)}^2_\LT ds \\ \leq (2\alpha + m + \delta) \int_{t_0}^{t} (s-t_0)^{2\alpha + m -1 + \delta} \norm{D^m \zz(\cdot,s)}^2_\LT, \quad \forall t > t_0,
\end{split}
\end{equation}	
for each $t_0 > t_*$ chosen appropriately large. Using the induction hypothesis (\ref{des_ind2}) for $k=m-1$, we obtain
\begin{equation}\label{sup_e8}
(t-t_0)^{2\alpha + m} \norm{D^m \zz (\cdot,t)}^2_\LT \leq \frac{1}{\delta [(2-\epsilon)\nu]^{m}} \bigg{[} \prod_{j=0}^{m} (2\alpha + j + \delta) \bigg{]} (\tilde{\lambda}_{0}(\alpha) + \epsilon)^2 .
\end{equation}
Moreover, from bounds (\ref{des_ind2}) and (\ref{sup_e7}), we also get
\begin{equation}\label{sup_e9}
\begin{split}
\int_{t_0}^{t} (s-t_0)^{2\alpha + m + \delta} \norm{D^{m+1} \zz (\cdot,s)}^2_\LT ds \\ \leq \bigg{(} \frac{1}{\delta [(2-\epsilon)\nu]^{m+1}} \bigg{[} \prod_{j=0}^{m} (2\alpha + j + \delta) \bigg{]} (\tilde{\lambda}_{0}(\alpha) + \epsilon)^2 \bigg{)}(t-t_0)^\delta.	
\end{split}
\end{equation}	
Since $\epsilon>0$ is arbitrary, this shows inequality \eqref{inequality_z} when $ \tilde{\lambda}_0(\alpha) < \infty$.
\begin{rem}
	Recalling inequalities \eqref{sup_e8} and \eqref{sup_e9}, we have shown that for every $ m \geq 0 $ and each $ \epsilon > 0 $ arbitrary, there exists $ t_{**}=t_{**} (m, \epsilon) > t_*$ sufficiently large so that 
	\begin{equation}\label{sup_e10} 
	\begin{split} 
	(t-t_0)^{2\alpha + m + \delta} \norm{D^m \zz(\cdot,t)}^2_\LT +  (2-\epsilon)\nu \int_{t_0}^{t} (s-t_0)^{2\alpha + m + \delta} \norm{D^{m+1} \zz(\cdot,s)}^2_\LT ds \\ \underbrace{+ 2 \chi \int_{t_0}^{t} (s-t_0)^{2\alpha + m + \delta} \norm{D^{m} \ww(\cdot,s)}^2_\LT ds}_{(i)} \\+ \underbrace{2 \kappa \int_{t_0}^{t} (s-t_0)^{2\alpha + m + \delta} \norm{D^{m} \zeta (\cdot,s)}^2_\LT ds}_{(ii)} \\\leq\frac{1}{\delta [(2-\epsilon)\nu]^{m}} \bigg{\{} \prod_{j=0}^{m} (2\alpha + j + \delta) \bigg{\}} (\tilde{\lambda}_{0}(\alpha) + \epsilon)^2 \, (t-t_0)^\delta, 
	\end{split} 
	\end{equation} 
$ \forall t \geq t_0 $, $ \forall t_0 \geq t_{**}(m,\epsilon) $. This suggests the following: the presence of the terms (i) and (ii)  on the LHS indicate that $ \| D^m \ww (\cdot,t) \|_{L^2(\RR^3)} $ and $ \| D^m \zeta (\cdot,t) \|_{L^2(\RR^3)} $ may decay faster (as $ t \to \infty $) than what is given by our general result \eqref{sup_e8} above. Indeed, in particular, inequality \eqref{sup_e8} implies that $ \|D^m \ww (\cdot,t)\|_{L^2} = O(t^{-\alpha - \frac{m}{2}}) $	and $ \|D^m \zeta (\cdot,t)\|_{L^2} = O(t^{-\alpha - \frac{m}{2} - \frac{1}{2}}) $ as $ t \to \infty $ for each $ m \geq 0 $. However, by inequality \eqref{sup_e10} they actually may decay faster when the vortex viscosity $ \chi $ is present. This is exactly the spirit of the next fundamental inequality \eqref{inequality_w} to be shown below. 
\end{rem}

\subsection{Proof of bound \eqref{inequality_w}}
As in the proof of inequality \eqref{eqn:Difference_estimate_w}, taking advantage of the linear damping term, using Lemma \ref{lema-lame-operator} and the Duhamel`s principle, one has
\begin{equation}\label{sup_des_chi_e3}
\begin{split}
t^{\alpha + \frac{m+1}{2}} \norm{D^m \ww(\cdot,t)}_\LL \leq \bigg{[} t^{\alpha + \frac{m+1}{2}} e^{-2\chi(t-t_0)}\norm{D^m e^{\mathbb{L}(t-t_0)}\ww(\cdot,t_0)}_\LL\\ 
\mbox{} + t^{\alpha + \frac{m+1}{2}} \int_{t_0}^{t} e^{-2\chi(t-s)}\norm{e^{\mathbb{L}(t-s)}D^m (\uu \cdot \nabla \ww)(\cdot,s)}_\LL ds 
\\\mbox{} + \chi t^{\alpha + \frac{m+1}{2}} \int_{t_0}^{t} e^{-2\chi(t-s)}\norm{e^{\mathbb{L}(t-s)}D^{m+1} \uu(\cdot,s)}_\LL ds \bigg{]}=:\mathcal{I}_1(t) + \mathcal{I}_2(t) + \mathcal{I}_3(t) .
\end{split}
\end{equation}
As already mentioned, in \cite{MR3906315} the following $ L^2  $ decay for the higher-order derivatives was proved:
\begin{equation*}
t^{k/2} \|D^k \zz \|_{L^2(\RR^3)} \to 0 \quad {\text{ as}} \quad t\to \infty, \,\, {\text{for each} } \,\, k\in \mathbb{Z}^+ . 
\end{equation*} 
This implies $\mathcal{I}_2(t)  \to 0 $ as $ t \to \infty $. Moreover, $ \mathcal{I}_1(t) $ clearly tends to $ 0 $ as $ t \to \infty $.
So, it remains to examine the last term $ \mathcal{I}_3(t) $. To this end, given $ \epsilon>0 $, choose $ t_0>t_{**}(m,\epsilon) $ large enough and fix $ 0<\sigma<1 $. By inequality \eqref{sup_e10}, we get     
\begin{equation*}
\begin{split}
 \mathcal{I}_3(t) \leq \chi t^{\alpha+\frac{m+1}{2}} \int_{t_0}^{t} e^{-2\chi(t-s)}\norm{e^{\gamma \Delta(t-s)}D^{m+1} \uu(\cdot,s)}_\LL ds \\
\leq \chi t^{\alpha+\frac{m+1}{2}} \int_{t_0}^{t} e^{-2\chi(t-s)}\norm{D^{m+1} \uu(\cdot,s)}_\LL ds \\ 
\leq \chi\,\nu^{-\frac{m+1}{2} } K_{\alpha,m+1} t^{\alpha+\frac{m+1}{2}} \int_{t_0}^{t} e^{-2\chi(t-s)} (\tilde{\lambda}_0(\alpha) + \epsilon) s^{-\alpha-\frac{m+1}{2}} ds \\
= \chi \nu^{-\frac{m+1}{2} } K_{\alpha,m+1}(\tilde{\lambda}_0(\alpha) + \epsilon) t^{\alpha+\frac{m+1}{2}} \bigg{\{} \int_{t_0}^{\sigma t} e^{-2\chi(t-s)} s^{-\alpha-\frac{m+1}{2}} ds \\+ \int_{\sigma t}^{t} e^{-2\chi(t-s)} s^{-\alpha-\frac{m+1}{2}} ds  \bigg{\}} \\
\leq\chi \nu^{-\frac{m+1}{2} } K_{\alpha,m+1}(\tilde{\lambda}_0(\alpha)+\epsilon) \bigg{\{} t^{\alpha+\frac{m+1}{2}}e^{-2\chi t(1-\sigma)} \bigg{(} 2\frac{t_0^{-\alpha-\frac{m-1}{2}} - (\sigma t)^{-\alpha-\frac{m-1}{2}}}{2\alpha + m -1} \bigg{)} \\ + \sigma^{-\alpha -\frac{m+1}{2}}\bigg{(} \frac{1- e^{-2\chi t (1-\sigma)}}{2\chi} \bigg{)} \bigg{\}}.
\end{split}
\end{equation*}
Letting $\epsilon \to 0^+$, $t \to \infty$ and $\sigma \to 1^-$ (in this order), one has
\begin{equation}\label{sup_des_chi_e8}
\limsup_{t \rightarrow \infty} \chi t^{\alpha+\frac{m+1}{2}} \int_{t_0}^{t} e^{-2\chi(t-s)}\norm{e^{\gamma \mathbb{L}(t-s)}D^{m+1} \uu}_\LL ds \leq \frac{K_{\alpha,m+1}}{2}  \,\nu^{-\frac{m+1}{2} } \,\tilde{\lambda}_0(\alpha).
\end{equation} 
This, together with inequality \eqref{sup_des_chi_e3}, leads to the desired inequality \eqref{inequality_w} in case $ \tilde{\lambda}_0(\alpha) < \infty $, i.e.,
\begin{equation}\label{sup_des_chi}
\limsup_{t \to \infty} t^{\alpha+\frac{m+1}{2}} \norm{D^m \ww(\cdot,t)}_{\LT} \leq  \frac{K_{\alpha,m+1}}{2}\;\;\nu^{-\frac{m+1}{2}} \;\; \tilde{\lambda}_0(\alpha).
\end{equation}	  
\qed 
\begin{claim}\label{rmk-remove-w}
The assumption $ \tilde{\lambda}_0(\alpha) < \infty $ in the proofs above is not necessary, since
\begin{equation}\label{assumption}
\text{ If }\,\,\,{\lambda}_0(\alpha) < \infty, \text{ then }\,\,t^\alpha \|\ww(\cdot,t) \|_{L^2} \to 0 \,\,\, \text{ as}\,\,\, t \to \infty, \,\,\text{i.e.},\,\,{\lambda}_0(\alpha) = \tilde{\lambda}_0(\alpha) . 
\end{equation} 
\end{claim}   
\begin{proof} 	
We divide the proof into three cases, depending on the value of $\alpha$. 

If $\alpha < 1/2$, then $t^\alpha \norm{\ww(\cdot,t)}_\LT \to 0$ as $t \to \infty$ (see convergence \eqref{ww_to_zero} for $ s =0 $) and $\tilde{\lambda}_0(\alpha) < \infty$. Thus, if $\alpha \leq 1/2$ , then inequality 
{\eqref{sup_des_chi}}  leads to inequalities \eqref{inequality}.
	
	If $1/2 \leq \alpha < 1$, then $t^{1/2} \norm{\ww(\cdot,t)}_\LT \to 0$ as $t \to \infty$.
	 Now, letting $m=0$ in inequality \eqref{sup_des_chi}, one concludes
	\begin{equation*}
	\limsup_{t \to \infty} t \norm{\ww(\cdot,t)}_\LT < \infty.
	\end{equation*}	
	This implies $\limsup_{t \to \infty} t^\alpha \norm{\ww(\cdot,t)}_\LT = 0$, that is, ${\lambda}_0(\alpha) = \tilde{\lambda}_0(\alpha)$, leading once more to the desired inequalities \eqref{inequality}.
	
	If $\alpha \geq 1$, just repeat step-by-step the same argument above. 
	In at most $2[\alpha] + 1$ steps, one gets the desired result. 
	\end{proof}
\section{Proof of Theorem \ref{thm2}}
First of all, note that it is enough to prove inequalities \eqref{error_inequality_z} 
and \eqref{error_inequality_w} for times $t_0 > t_*$, the regularity time. Indeed, if the inequalities 
\eqref{error_inequality_z} 
and \eqref{error_inequality_w} hold for $t_0 > t_*$, then given $\hat{t} \in [0 , t_*]$, Theorem  
\ref{thm_linear_aprox} gives 
\begin{equation*}
\begin{split}
\limsup_{t \rightarrow \infty} t^{\alpha + \beta + \frac{m}{2}} \|\Ez(\cdot,t;\hat{t}) \|_{L^2} \leq 
\limsup_{t \rightarrow \infty} t^{\alpha + \beta + \frac{m}{2}} \|\Ez(\cdot,t;t_0) \|_{L^2}  \\
+ \limsup_{t \rightarrow \infty} t^{\alpha + \beta + \frac{m}{2}} \|\Ez(\cdot,t;t_0) - \Ez(\cdot,t;\hat{t})  \|_{L^2} = \limsup_{t \rightarrow \infty} t^{\alpha + \beta + \frac{m}{2}} \|\Ez(\cdot,t;t_0) \|_{L^2} < C_{\alpha,m} .
\end{split}
\end{equation*}  
The same argument works for $ \Ew $. So, from now on, consider $t_0 > t_*$. 
\subsection{Proof of inequality t\eqref{error_inequality_z}}
Using Lemma \ref{lema-linear_opertaror} and the integral representations of $ \zz $ and $ \bar{\zz} $, write
\begin{equation}\label{duhamel_error_Dm_e1}
\begin{split}
\norm{D^m \Ez(\cdot,t;t_0)}_\LT = \norm{D^m \left[ \zz(\cdot,t) - e^{\AAA (t-t_0)} \zz(\cdot,t_0) \right]}_\LT \\ \leq \int_{t_0}^{t} \norm{e^{c\, \Delta (t-s)} D^m \left[ \uu \cdot \nabla \uu \right]}_\LT + \norm{e^{c\, \Delta (t-s)} D^m \left[ \uu \cdot \nabla \ww \right]}_\LT ds.
\end{split}
\end{equation}  
We estimate in details the first term in the right hand side of inequality (\ref{duhamel_error_Dm_e1}), since the analysis for the second one is completely similar \footnote{Actually, the second term decays even faster than the first one due to inequality \eqref{inequality_w}}. To this end, take $t > 2 t_0$ and split the domain of integration as 
$ [t_0, \frac{t}{2} ] \cup [\frac{t}{2},t ] $. We begin by examining the integral on the interval $[\frac{t}{2},t ] $: 
\begin{equation*}
\begin{split}
t^{\alpha + \beta + m/2} \int_{t/2}^{t} \norm{e^{c\, \Delta(t-s)}D^m (\uu \cdot \nabla \uu)(\cdot,s)}_\LL ds \\ \leq
t^{\alpha + \beta + m/2} \int_{t/2}^{t} \sum_{l=0}^{m} \norm{D^l \uu (\cdot,s)}_{L^4(\RR^3)} \norm{D^{m-l+1} \uu (\cdot,s)}_{L^4(\RR^3)} ds \\ \leq  t^{\alpha + \beta+m/2} \int_{t/2}^{t} \sum_{l=0}^{m} \norm{D^l \uu(\cdot,s)}_\LL^{1/4} \norm{D^{l+1} \uu(\cdot,s)}_\LL^{3/4} \\ \times \norm{D^{m-l+1} \uu(\cdot,s)}_\LL^{1/4} \norm{D^{m-l+2} \uu(\cdot,s)}_\LL^{3/4} ds \\ \leq K \, \nu^{-(\frac{m}{2}+\frac{5}{4})} \, (\lambda_0(\alpha) + \epsilon)^2			
t^{\alpha + \beta+m/2} \int_{t/2}^{t} \sum_{l=0}^{m} s^{-\frac{\alpha}{4}-\frac{l}{8}}  s^{-\frac{3\alpha}{4}-\frac{3(l+1)}{8}}  s^{-\frac{\alpha}{4}-\frac{m-l+1}{8}} s^{-\frac{3\alpha}{4}-\frac{3(m-l+2)}{8}} ds \\ \leq K \, \nu^{-(\frac{m}{2}+\frac{5}{4})} \, (\lambda_0(\alpha) + \epsilon)^2			
t^{\alpha + \beta + m/2} \int_{t/2}^{t} s^{-2\alpha} s^{-\frac{2m+5}{4}} ds  \leq K (\lambda_0(\alpha) + \epsilon)^2,
\end{split}
\end{equation*}
where
\begin{equation*}
K = \sum_{l=0}^{m} K_{\alpha,l}^{1/4} K_{\alpha,l+1}^{3/4} K_{\alpha,m-l+1}^{1/4} K_{\alpha,m-l+2}^{3/4} 
\end{equation*}
and
\[
	K_{\alpha,m} =  \min_{\delta>0} \Bigg\{ \delta^{-1/2} \prod_{j=0}^{m} (\alpha + \frac{j}{2} +\delta)^{1/2}  \Bigg\}. 	\]
Note that to bound the $L^4$ norms in the inequality above, we have used the classical Sobolev-Nirenberg-Gagliardo inequality 
$ \|f\|_{L^4} \leq \|f \|^{1/4}_{L^2} \|Df \|^{3/4}_{L^2} $, for any $ f \in H^1 $. 

It remains to bound the integral over $[t_0, \frac{t}{2}]$. Since, $\|e^{\Delta \tau} D^m( \uu \cdot \nabla \uu)\|_{L^2} \leq \tau^{-\frac{5}{4} - \frac{m}{2}} \|\uu \|^2_{L^2}$ (see  \cite{MR760047,MR1994780}), we have
\begin{equation*}
\begin{split}
\int_{t_0}^{\frac{t}{2}} \norm{e^{c\Delta (t-s)} D^m \left[ \uu \cdot \nabla \uu \right]}_\LT ds \leq c^{-(\frac{m}{2}+\frac{5}{4})} \int_{t_0}^{\frac{t}{2}} (t-s)^{-(\frac{m}{2} + \frac{5}{4})} \norm{\uu(\cdot,s)}_\LT^2 ds.
\end{split}
\end{equation*}
We divide the analysis of the last integral into three cases depending on the value of $\alpha$: 


\noindent Case 1: $\alpha \in (0,1/2)$. Use that $\norm{\uu(\cdot,s)}_\LT^2 < (\lambda_0(\alpha)+\epsilon)^2 s^{-2\alpha}$ to get
\begin{equation*}
\begin{split}
t^{\alpha + \beta + m/2} \int_{t_0}^{t/2} \norm{e^{c\, \Delta (t-s)} D^m \left[ \uu \cdot \nabla \uu \right]}_\LT ds\\ \leq c^{-(\frac{m}{2}+\frac{5}{4})}\,(\lambda_0(\alpha)+\epsilon)^2 t^{2\alpha + 1/4 + m/2} \int_{t_0}^{t/2} (t-s)^{-5/4 -m/2}s^{-2\alpha} ds\\ \leq c^{-(\frac{m}{2}+\frac{5}{4})}\, (\lambda_0(\alpha)+\epsilon)^2 2^{2\alpha + 1/4+m/2} = c^{-(\frac{m}{2}+\frac{5}{4})}\, (\lambda_0(\alpha)+\epsilon)^2 \, 2^{\alpha + \beta +m/2}
\end{split}.
\end{equation*}

\noindent Case 2: $\alpha \in (1/2,1)$. Use that $\norm{\uu(\cdot,s)}_\LT^2 < \norm{\zz_0}(\lambda_0(\alpha)+\epsilon) s^{-\alpha}$ to get
\begin{equation*}
\begin{split}
t^{\alpha + \beta + m/2} \int_{t_0}^{t/2} \norm{e^{c\, \Delta (t-s)} D^m \left[ \uu \cdot \nabla \uu \right]}_\LT ds\\ \leq c^{-(\frac{m}{2}+\frac{5}{4})}\, \norm{\zz_0}(\lambda_0(\alpha)+\epsilon) t^{\alpha + 1/4 + m/2} \int_{t_0}^{t/2} (t-s)^{-5/4 -m/2}s^{-\alpha} ds\\
 \leq c^{-(\frac{m}{2}+\frac{5}{4})}\, \norm{\zz_0}(\lambda_0(\alpha)+\epsilon) 2^{\alpha + \beta +m/2}.
\end{split}
\end{equation*}

\noindent Case 3: $\alpha \in (1,5/4)$. Using the fact that $\alpha > 1$ implies $\int_{0}^{\infty}\norm{\uu(\cdot,s)}_\LT ds < \infty$, we get
\begin{equation*}
\begin{split}
t^{\alpha + \beta + m/2} \int_{t_0}^{t/2} \norm{e^{c\, \Delta (t-s)} D^m \left[ \uu \cdot \nabla \uu \right]}_\LT ds\\ \leq c^{-(\frac{m}{2}+\frac{5}{4})}\, \norm{\zz_0} t^{5/4 + m/2} \int_{t_0}^{t/2} (t-s)^{-5/4 -m/2}\norm{\uu(\cdot,s)}_\LT  ds \\
 \leq c^{-(\frac{m}{2}+\frac{5}{4})}\, \norm{\zz_0} 2^{5/4 +m/2} = c^{-(\frac{m}{2}+\frac{5}{4})}\, \norm{\zz_0} 2^{\alpha + \beta +m/2}.
\end{split}
\end{equation*}
This concludes the proof of  inequality \eqref{error_inequality_z}. \qed 
\subsection{Proof of inequality \eqref{error_inequality_w}}
Using Lemma \ref{lema-lame-operator}, 
we get, for $ \Ew(\cdot,t)$, the bound
\begin{equation}\label{eqn:preparation-w}
\begin{split}
t^{\alpha + \beta + \frac{m+1}{2}} \norm{D^m \Ew(\cdot,t)}_\LL \leq   t^{\alpha+ \beta + \frac{m+1}{2}} e^{-2\chi(t-t_0)}\norm{D^m e^{c\, \Delta(t-t_0)}\Ew(\cdot,t_0)}_\LL\\ 
+ t^{\alpha + \beta+ \frac{m+1}{2}} \int_{t_0}^{t} e^{-2\chi(t-s)}\norm{e^{c\, \Delta(t-s)}D^m (\uu \cdot \nabla \ww)(\cdot,s)}_\LL ds \\
+ \chi t^{\alpha + \beta+ \frac{m+1}{2}} \int_{t_0}^{t} e^{-2\chi(t-s)}\norm{e^{c\, \Delta(t-s)}D^{m+1} \Eu(\cdot,s)}_\LL ds.
\end{split}
\end{equation}
As in the proof of inequality \eqref{inequality_w} (see bound \eqref{sup_des_chi_e3}), all the terms in the right hand side of inequality \eqref{eqn:preparation-w} tend to $0$ as $t$ tends to infinity except for the last one. Hence, this is the term we will focus our attention. By inequality \eqref{error_inequality_z}, one has
\begin{equation*}
\begin{split}
\chi t^{\alpha+\beta+\frac{m+1}{2}} \int_{t_0}^{t} e^{-2\chi(t-s)}\norm{e^{\gamma \Delta(t-s)}D^{m+1} \Eu(\cdot,s)}_\LL ds \\
\leq \chi t^{\alpha+\beta+\frac{m+1}{2}} \int_{t_0}^{t} e^{-2\chi(t-s)}\norm{D^{m+1} \Eu(\cdot,s)}_\LL ds \\ 
\leq (C(\nu,\chi,\alpha,m+1)+\epsilon)\chi t^{\alpha+\beta+\frac{m+1}{2}} \int_{t_0}^{t} e^{-2\chi(t-s)}  s^{-\alpha -\beta -\frac{m+1}{2}} ds \\
= (C(\nu,\chi,\alpha,m+1)+\epsilon)\chi t^{\alpha+\beta+\frac{m+1}{2}} \bigg{[} \int_{t_0}^{\sigma t} e^{-2\chi(t-s)} s^{-\alpha -\beta-\frac{m+1}{2}} ds \\
+ \int_{\sigma t}^{t} e^{-2\chi(t-s)} s^{-\alpha-\beta-\frac{m+1}{2}} ds  \bigg{]} \\
\leq (C(\nu,\chi,\alpha,m+1)+\epsilon)\chi t^{\alpha+\beta+\frac{m+1}{2}} \bigg{[} e^{-2\chi t(1-\sigma)}\int_{t_0}^{{\sigma } t}  s^{-\alpha-\beta-\frac{m+1}{2}} ds \\+ ({\sigma } t)^{-\alpha-\beta-\frac{m+1}{2}}\int_{{\sigma } t}^{t} e^{-2\chi(t-s)} ds  \bigg{]} \\
\leq (C(\nu,\chi,\alpha,m+1)+\epsilon)\chi \bigg{[} t^{\alpha+\beta+\frac{m+1}{2}}e^{-2\chi t(1-\sigma)} \bigg{(} 2\frac{t_0^{-\alpha-\beta-\frac{m-1}{2}} - ({\sigma } t)^{-\alpha-\beta-\frac{m-1}{2}}}{2\alpha +2\beta + m -1} \bigg{)} \\ + \sigma^{-\alpha-\beta -\frac{m+1}{2}}\bigg{(} \frac{1- e^{-2\chi t (1-\sigma)}}{2\chi} \bigg{)} \bigg{]}.
\end{split}
\end{equation*}
Letting, in this order, $\epsilon \to 0^+$, taking the $\limsup$ as $t \to \infty$ and letting ${\sigma } \to 1^-$, we get
\begin{equation*}
\limsup_{t \rightarrow \infty} t^{\alpha+\beta+\frac{m+1}{2}} \int_{t_0}^{t} e^{-2\chi(t-s)}\norm{e^{\gamma \Delta(t-s)}D^{m+1} \Eu(\cdot,s)}_\LL ds \leq \frac{1}{2} C(\nu,\chi,\alpha,m+1).
\end{equation*}
which leads to the desired inequality \eqref{error_inequality_w}. \qed
\section{Additional comments and the fundamental inequality for other related systems}
We present some direct consequences and extensions of the results previous results. 
We also apply the previous arguments to quickly obtain results for other related systems.  
\subsection{Direct consequences and extensions}
From the results above, we can immediately get, through standard interpolations, some general inequalities.
Precisely,
\begin{equation*}
	\limsup_{t \to \infty} t^{\alpha + \frac{s}{2}} \norm{\zz}_{\dot{H}^s(\RR^3)} \leq \min_{m>s}\, [\,K_{\alpha,m}\,]^{\frac{s}{m}} \nu^{-\frac{m}{s}} \lambda_{0}(\alpha),
\end{equation*}
\begin{equation*}
	\limsup_{t \to \infty} t^{\alpha + \frac{m}{2} + \frac{3}{4}} \norm{D^m\zz}_{L^\infty(\RR^3)} \leq ( K_{\alpha,m}\,\nu^{-\frac{m}{2}})^{1/4} ( K_{\alpha,m+2}\, \nu^{-\frac{m}{2}-1})^{3/4}\, \lambda_{0}(\alpha) ,
\end{equation*}
and
\begin{equation*}
	\limsup_{t \to \infty} t^{\alpha + \frac{m}{2} + \frac{3(p-2)}{4p}} \norm{D^m \zz}_{L^p(\RR^3)} \leq C_{\nu,m,\alpha,p}\,\lambda_{0}(\alpha) \mbox{,\,\, for all } 2 \leq p < \infty.	
\end{equation*}
\begin{rem}	
	Similar versions of the inequalities above can be obtained for $\Ez$ and $\Ew$.	It is also important to note that the main theorems \ref{thm1} and \ref{thm2} imply the following interesting decays:
	\begin{subequations}
		\begin{equation*}
		\norm{D_t^k D_x^m \zz(\cdot,t)}_\LP = O(t^{-\alpha -k -\frac{m}{2} -\frac{3(p-2)}{4p}}),
		\end{equation*}
		\begin{equation*}
		\norm{D_t^k D_x^m \ww(\cdot,t)}_\LP = O(t^{-\alpha -k -\frac{m+1}{2} -\frac{3(p-2)}{4p}}),
		\end{equation*}
		\begin{equation*}
		\norm{D_t^k D_x^m \Ez(\cdot,t)}_\LP = O(t^{-\alpha -\beta -k -\frac{m}{2} -\frac{3(p-2)}{4p}}),
		\end{equation*}
		\begin{equation*}
		\norm{D_t^k D_x^m \Ew(\cdot,t)}_\LP = O(t^{-\alpha - \beta -k -\frac{m+1}{2} -\frac{3(p-2)}{4p}}),
		\end{equation*}
		where $2 \leq p \leq \infty$ and $\beta$ is as in Theorem \ref{thm2}.
	\end{subequations}
\end{rem}

\subsection{Other Systems}\label{other_systems}
As mentioned in the introduction, the technique presented can be used to obtain similar estimates for other dissipative type problems. Let us now give some examples.
\subsubsection{Dissipative linear system}
Consider 
\begin{equation}\label{eq0}
v_{t}  = \LLL v
\end{equation}
where $\LLL : \RR^3 \to (L^2(\RR^3))^3$ is a pseudodifferential operator with symbol $M(\xi)$ such that
\begin{equation*}
M(\xi) = P^{-1}(\xi)D(\xi)P(\xi), \mbox{ \,\,\,\,} \xi-a.e.,
\end{equation*}
where $P(\xi) \in O(3)$ and $D(\xi)_{ij} = -c_i |\xi|^{2\gamma} \delta_{ij}$, 
for $c_i>c>0$ and $0<\gamma \leq 1$. For a study on the decay rates of solutions for these types of equations involving the {\it decay character}, see \cite{MR3493117,MR2493562, MR3355116}. In fact, this approach assures that condition $\lambda_0 ( \alpha ) < \infty$ (condition \eqref{eqn:lambda-assumption})
for certain $\alpha>0$ depending on the initial data. 

Let $\Lambda = (-\Delta)^{1/2}$. Applying $ \Lambda^{m\gamma}$ in equation \eqref{eq0} and taking the dot product with $(t-t_0)^{2\alpha + m + \delta} \Lambda^{m\gamma} v$, we get
\begin{equation}
\begin{split}
	\frac{1}{2}(t-t_0)^{2\alpha + m + \delta} \dfrac{d}{dt} \norm{v(t)}^2_{\dot{H}^{m \gamma}} = (t-t_0)^{2\alpha + m + \delta}  \langle |\xi|^{m\gamma} \hat{v}, M |\xi|^{m\gamma} \hat{v} \rangle_{L^2} \\
	= -(t-t_0)^{2\alpha + m + \delta} \langle (-D)^{1/2} P |\xi|^{m\gamma} \hat{v}, (-D)^{1/2} P |\xi|^{m\gamma} \hat{v} \rangle_{L^2} \\
	= -(t-t_0)^{2\alpha + m + \delta} \int_{\RR^n} |(-D)^{1/2} P |\xi|^{m\gamma} \hat{v}|^2 d\xi \\\leq -C (t-t_0)^{2\alpha + m + \delta}\int_{\RR^n} |\xi|^{2(m+1)\gamma} |\hat{v}|^2 d\xi = -C (t-t_0)^{2\alpha + m + \delta}\norm{v(t)}^2_{\dot{H}^{(m+1)\gamma}}.
\end{split}
\end{equation}
Therefore, 
\begin{equation}
\begin{split}
	(t-t_0)^{2\alpha + m + \delta} \norm{v(t)}^2_{\dot{H}^{m \gamma}} + 2C \int_{t_0}^{t} (t-t_0)^{2\alpha + m + \delta}\norm{v(s)}^2_{\dot{H}^{(m+1)\gamma}} ds\\
	 \leq (2\alpha + m + \delta) \int_{t_0}^{t} (t-t_0)^{2\alpha + m -1+ \delta} \norm{v(s)}^2_{\dot{H}^{m \gamma}}ds.
\end{split}
\end{equation}
Using the same induction argument as in the proof of Theorem \ref{thm1}, one has
\begin{equation}\label{linear_dissipative_inequality}
	\limsup_{t \rightarrow \infty} t^{\alpha + m/2} \norm{v(t)}^2_{\dot{H}^{m \gamma}} \leq K_{\alpha,m} C^{-m/2}  \limsup_{t \rightarrow \infty} t^{\alpha} \norm{v(t)}_\LT
\end{equation}
where
\begin{equation}
K_{\alpha,m} =  \min_{\delta>0} \Bigg\{ \delta^{-1/2} \prod_{j=0}^{m} (\alpha + \frac{j}{2} +\delta)^{1/2}  \Bigg\}.
\end{equation}
\subsubsection{Tropical climate model equations}

Consider the system
\begin{subequations}\label{tropical_eq}
	\begin{equation}
		\uu_t + (\uu \cdot \nabla)\uu - \mu \Delta \uu + \nabla p + \nabla \cdot (\vv \otimes \vv) = 0,
	\end{equation}
	\begin{equation}
		\vv_t + (\uu \cdot \nabla) \vv - \nu \Delta \vv + \nabla \theta + (\vv \cdot \nabla)\uu = 0,
	\end{equation}
	\begin{equation}
		\theta_t + (\uu \cdot \nabla)\theta - \eta \Delta \theta + \nabla \cdot \vv = 0,
	\end{equation}
	\begin{equation}
		\nabla \cdot \uu = 0,
	\end{equation}
\end{subequations}
where $\uu = (u_1(x, t), u_2(x, t), u_3(x, t))$ and $\vv = (v_1(x, t), v_2(x, t), v_3(x, t))$ are the barotropic mode and the first baroclinic mode of the velocity, respectively; $\theta = \theta(x, t)$ and $p = p(x, t)$ represent the scalar temperature and scalar pressure, respectively. The coefficients $\mu$, $\nu$, and $\eta$ are nonnegative.
System \eqref{tropical_eq} was first introduced in \cite{MR2119930} without any dissipation terms ($\mu = \nu = \eta = 0$). For a study on its regularity and well-posedness, see \cite{MR3479523, MR2342696, MR2890308, MR3237881, MR3264417, MR3455664, MR3949522}.

Following the same techniques presented in the proof of Theorem \ref{thm1} one can obtain the inequality
\begin{equation}
	\begin{split}
		(t-t_0)^{2\alpha + m + \delta} \norm{(D^m \uu, D^m \vv, D^m \theta)(\cdot,t)}_\LT^2 +\\ (2-\epsilon)\min{(\mu,\nu,\eta)} \int_{t_0}^{t} (s - t_0)^{2\alpha + m+ \delta} \norm{(D^{m+1} \uu, D^{m+1} \vv, D^{m+1} \theta)(\cdot,s)}_\LT^2 ds \\
		\leq (2\alpha + m + \delta) \int_{t_0}^{t} (s-t_0)^{2\alpha + m + \delta} \norm{(D^m \uu, D^m \vv, D^m \theta)(\cdot,s)}_\LT^2,
	\end{split}
\end{equation}
for each integer $m \geq 0$. Through the same induction argument as in the proof of Theorem \ref{thm1}, one gets
\begin{equation}
	\begin{split}
		\limsup_{t \to \infty}\, t^{\alpha+\frac{m}{2}} \norm{(D^m \uu, D^m \vv, D^m \theta)(\cdot,t)}_\LT\\ \leq  K_{\alpha,m} (\min\{\mu,\nu,\eta\})^{-m/2}\;\; 	\limsup_{t \to \infty}\, t^{\alpha} \norm{(\uu,\vv,\theta)(\cdot,t)}_\LT,
	\end{split}
\end{equation}
where
\begin{equation}
	K_{\alpha,m} =  \min_{\delta>0} \Bigg\{ \delta^{-1/2} \prod_{j=0}^{m} (\alpha + \frac{j}{2} +\delta)^{1/2}  \Bigg\}.
\end{equation}

\subsubsection{Navier-stokes with rotational framework}
Consider the system 
\begin{equation}\label{eq_coriolis}
\begin{split}
	\uu_t + (\uu \cdot \nabla)\uu + \omega \mathbb{J} \uu + \nabla p = & \mu \Delta \uu \mbox{ \,\,\,\, in } \RR^3 \times (0, \infty) ,  \\
	\nabla \cdot \uu =& 0 \mbox{ \,\,\,\,\,\,\,\,\,\,\,\,\,\, in } \RR^3 \times (0, \infty) , \\
	\uu(\cdot,0) =&\uu_0 \in L^2_{\sigma} (\RR^3) ,
\end{split}
\end{equation}
where $ \omega \in \RR $ is the Coriolis parameter and $ \mathbb{J}\uu := {\bf e}_3 \wedge \uu $ is the Coriolis force acting on the fluid. Observe that $\mathbb{J}u = (-u_2,u_1,0)$ and therefore this term does not affect the energy inequality obtained for $\uu$. Following the same procedure as in the proof of Theorem \ref{thm1}, we get
\begin{equation}
\begin{split}
	(t-t_0)^{2\alpha + m + \delta} \norm{D^m \uu(\cdot,t)}^2_\LT +  (2-\epsilon)\nu \int_{t_0}^{t} (s-t_0)^{2\alpha + m + \delta} \norm{D^{m+1} \uu(\cdot,s)}^2_\LT ds \\ \leq (2\alpha + m + \delta) \int_{t_0}^{t} (s-t_0)^{2\alpha + m -1 + \delta} \norm{D^m \uu(\cdot,s)}^2_\LT,
\end{split}
\end{equation}
for each integer $ m \geq 0$. Using the same induction argument as in the proof of Theorem \ref{thm1}, one has
\begin{equation}
\begin{split}
	\limsup_{t \to \infty}\, t^{\alpha+\frac{m}{2}} \norm{D^m \uu(\cdot,t)}_\LT \leq  K_{\alpha,m} \mu^{-m/2}\;\; 	\limsup_{t \to \infty}\, t^{\alpha} \norm{\uu(\cdot,t)}_\LT
\end{split}
\end{equation}
where
\begin{equation}
K_{\alpha,m} =  \min_{\delta>0} \Bigg\{ \delta^{-1/2} \prod_{j=0}^{m} (\alpha + \frac{j}{2} +\delta)^{1/2}  \Bigg\}.
\end{equation}

For the existence of a global mild solution assuming small data in $ H^{1/2}$, see \cite{MR2609321}.
The equations \eqref{eq_coriolis} are also related to the Navier-Stokes flow past rotating obstacles (see e.g.  \cite{MR2042672}). 
\subsubsection{Magneto-micropolar equations} 
When the magnetic field is taken into account in system \eqref{micropolar}, we get the so called magneto-micropolar system
\begin{equation}
\label{eqn:MHDmicropolar}
\left\{
\begin{aligned}    
\partial_t \uu +( \uu \cdot \nabla) \uu + \nabla p& =  (\mu + \chi) \Delta \uu + \chi \nabla \times \ww + (\bb \cdot \nabla)  \bb,  \\
\partial_t \ww + (\uu \cdot \nabla) \ww & =  \gamma \Delta \ww +  \nabla (\nabla \cdot \ww) + \chi \nabla \times \uu -2 \chi \ww,  \\
\partial_t \bb + (\uu \cdot \nabla) \bb & =  \nu \Delta \bb + (\bb \cdot \nabla) \uu, \\
\nabla \cdot \uu (\cdot,t) & =  \nabla \cdot \bb(\cdot,t)\, = 0,
\end{aligned}
\right.
\end{equation}
with initial data 
$ \zz_0 =(\:\!\uu_0, {\bf w}_0, \bb_0) \in L^2 _{\sigma}(\mathbb{R}^{3}) \!\times\! L^2 (\mathbb{R}^{3}) \times L^2 _{\sigma}(\mathbb{R}^{3}) $. The results on eventual smoothness, existence and uniqueness for the micropolar fluids system also hold for the magneto-micropolar system \eqref{eqn:MHDmicropolar}. One can derive results completely analgous to Theorems \ref{thm1} and \ref{thm2}, where, in this case, 
$ \zz(\cdot,t):= (\uu,\ww,\bb)(\cdot,t)$.   
\subsubsection{Other  Parabolic Systems} Let $f_j: \RR^n \to \RR^n$, $n=2,3,4$, $1 \leq j \leq N$, denote smooth vector fields with symmetric Jacobian $A_j(u) = D_u f_j(u)$. Assume $|f_j(u)| \leq C|u|^2$. Consider the parabolic system 
\begin{equation}
	u_t +\sum_{j=1}^{N}A_j(u)D_j u + D_j(A_j(u)u) = Pu
\end{equation}
where the first order terms have anti-symmetric form and the linear constant coefficient operator $P$ satisfies
\begin{equation*}
	P = \sum_{j=1}^{N} R_j D_j + \sum_{i,j=1}^{N} V_{ij}D_i D_j,
\end{equation*}
\begin{equation*}
	\hat{P}(k) + \hat{P}^*(k) \leq -c|k|^2 I , 
\end{equation*}
for some $c>0$. Since

\begin{equation*}
	\int u^T D_j (A_j(u)u)dx = - \int (D_j u^T) A_j(u) u dx
\end{equation*}
and $A_j = A^T_j$, an analogous reasoning as used in Theorem \ref{thm1} can be used to show that
\begin{equation*}
	\limsup_{t \to \infty} t^{\alpha +m/2} \norm{D^m u(\cdot,t)}_{L^2(\RR^n)} \leq\,C \lambda_0(\alpha),
\end{equation*}
where $ \lambda_0 (\alpha) :=\limsup_{t \to \infty} t^{\alpha} \norm{u(\cdot,t)}_{L^2(\RR^n)} $.



\bibliographystyle{plain}
\bibliography{BrazGuterresPerusatoZinganoBiblio}{}

\begin{thebibliography}{10}

\bibitem{MR2493562}
Clayton Bjorland and Maria~E. Schonbek.
\newblock Poincar\'{e}'s inequality and diffusive evolution equations.
\newblock {\em Adv. Differential Equations}, 14(3-4):241--260, 2009.

\bibitem{MR2646523}
J.~L. Boldrini, M.~Dur\'{a}n, and M.~A. Rojas-Medar.
\newblock Existence and uniqueness of strong solution for the incompressible
  micropolar fluid equations in domains of {$\Bbb R^3$}.
\newblock {\em Ann. Univ. Ferrara Sez. VII Sci. Mat.}, 56(1):37--51, 2010.

\bibitem{MR3493117}
Lorenzo Brandolese.
\newblock Characterization of solutions to dissipative systems with sharp
  algebraic decay.
\newblock {\em SIAM J. Math. Anal.}, 48(3):1616--1633, 2016.

\bibitem{MR3955606}
P.~Braz~e Silva, F.~W. Cruz, L.~B.~S. Freitas, and P.~R. Zingano.
\newblock On the {$L^2$} decay of weak solutions for the 3{D} asymmetric fluids
  equations.
\newblock {\em J. Differential Equations}, 267(6):3578--3609, 2019.

\bibitem{MR3516831}
P.~Braz~e Silva, L.~Friz, and M.~A. Rojas-Medar.
\newblock Exponential stability for magneto-micropolar fluids.
\newblock {\em Nonlinear Anal.}, 143:211--223, 2016.

\bibitem{MR3907942}
P.~Braz~e Silva, J.~P. Zingano, and P.~R. Zingano.
\newblock A note on the regularity time of {L}eray solutions to the
  {N}avier-{S}tokes equations.
\newblock {\em J. Math. Fluid Mech.}, 21(1):Paper No. 8, 7, 2019.

\bibitem{MR673830}
Luis Caffarelli, Robert Kohn, and Louis Nirenberg.
\newblock Partial regularity of suitable weak solutions of the
  {N}avier-{S}tokes equations.
\newblock {\em Comm. Pure Appl. Math.}, 35(6):771--831, 1982.

\bibitem{MR3264417}
Chongsheng Cao, Jinkai Li, and Edriss~S. Titi.
\newblock Global well-posedness of strong solutions to the 3{D} primitive
  equations with horizontal eddy diffusivity.
\newblock {\em J. Differential Equations}, 257(11):4108--4132, 2014.

\bibitem{MR3237881}
Chongsheng Cao, Jinkai Li, and Edriss~S. Titi.
\newblock Local and global well-posedness of strong solutions to the 3{D}
  primitive equations with vertical eddy diffusivity.
\newblock {\em Arch. Ration. Mech. Anal.}, 214(1):35--76, 2014.

\bibitem{MR2342696}
Chongsheng Cao and Edriss~S. Titi.
\newblock Global well-posedness of the three-dimensional viscous primitive
  equations of large scale ocean and atmosphere dynamics.
\newblock {\em Ann. of Math. (2)}, 166(1):245--267, 2007.

\bibitem{MR2890308}
Chongsheng Cao and Edriss~S. Titi.
\newblock Global well-posedness of the {$3D$} primitive equations with partial
  vertical turbulence mixing heat diffusion.
\newblock {\em Comm. Math. Phys.}, 310(2):537--568, 2012.

\bibitem{MR2860636}
Qionglei Chen and Changxing Miao.
\newblock Global well-posedness for the micropolar fluid system in critical
  {B}esov spaces.
\newblock {\em J. Differential Equations}, 252(3):2698--2724, 2012.

\bibitem{MR167060}
Duane~W. Condiff and John~S. Dahler.
\newblock Fluid mechanical aspects of antisymmetric stress.
\newblock {\em Phys. Fluids}, 7:842--854, 1964.

\bibitem{MR3977513}
Felipe~W. Cruz.
\newblock Global strong solutions for the incompressible micropolar fluids
  equations.
\newblock {\em Arch. Math. (Basel)}, 113(2):201--212, 2019.

\bibitem{MR0204005}
A.~Cemal Eringen.
\newblock Theory of micropolar fluids.
\newblock {\em J. Math. Mech.}, 16:1--18, 1966.

\bibitem{MR2119930}
Dargan M.~W. Frierson, Andrew~J. Majda, and Olivier~M. Pauluis.
\newblock Large scale dynamics of precipitation fronts in the tropical
  atmosphere: a novel relaxation limit.
\newblock {\em Commun. Math. Sci.}, 2(4):591--626, 2004.

\bibitem{MR2042672}
Giovanni~P. Galdi.
\newblock Steady flow of a {N}avier-{S}tokes fluid around a rotating obstacle.
\newblock volume~71, pages 1--31. 2003.
\newblock Essays and papers dedicated to the memory of Clifford Ambrose
  Truesdell III, Vol. II.

\bibitem{MR0467030}
Giovanni~P. Galdi and Salvatore Rionero.
\newblock A note on the existence and uniqueness of solutions of the micropolar
  fluid equations.
\newblock {\em Internat. J. Engrg. Sci.}, 15(2):105--108, 1977.

\bibitem{MR3906315}
R.~H. Guterres, W.~G. Melo, J.~R Nunes, and C.~F. Perusato.
\newblock On the large time decay of asymmetric flows in homogeneous {S}obolev
  spaces.
\newblock {\em J. Math. Anal. Appl.}, 471(1-2):88--101, 2019.

\bibitem{MR3854334}
R.~H. Guterres, J.~R. Nunes, and C.~F. Perusato.
\newblock On the large time decay of global solutions for the micropolar
  dynamics in {$L^2(\Bbb{R}^n)$}.
\newblock {\em Nonlinear Anal. Real World Appl.}, 45:789--798, 2019.

\bibitem{MR4021907}
T.~Hagstrom, J.~Lorenz, J.~P. Zingano, and P.~R. Zingano.
\newblock On two new inequalities for {L}eray solutions of the
  {N}avier-{S}tokes equations in {$\Bbb{R}^n$}.
\newblock {\em J. Math. Anal. Appl.}, 483(1):123601, 10, 2020.

\bibitem{MR943824}
John~G. Heywood.
\newblock Epochs of regularity for weak solutions of the {N}avier-{S}tokes
  equations in unbounded domains.
\newblock {\em Tohoku Math. J. (2)}, 40(2):293--313, 1988.

\bibitem{MR2609321}
Matthias Hieber and Yoshihiro Shibata.
\newblock The {F}ujita-{K}ato approach to the {N}avier-{S}tokes equations in
  the rotational framework.
\newblock {\em Math. Z.}, 265(2):481--491, 2010.

\bibitem{MR760047}
Tosio Kato.
\newblock Strong {$L^{p}$}-solutions of the {N}avier-{S}tokes equation in
  {${\bf R}^{m}$}, with applications to weak solutions.
\newblock {\em Math. Z.}, 187(4):471--480, 1984.

\bibitem{MR1994780}
Heinz-Otto Kreiss, Thomas Hagstrom, Jens Lorenz, and Paulo Zingano.
\newblock Decay in time of incompressible flows.
\newblock {\em J. Math. Fluid Mech.}, 5(3):231--244, 2003.

\bibitem{MR1555394}
Jean Leray.
\newblock Sur le mouvement d'un liquide visqueux emplissant l'espace.
\newblock {\em Acta Math.}, 63(1):193--248, 1934.

\bibitem{MR3949522}
Hongmin Li and Yuelong Xiao.
\newblock Decay rate of unique global solution for a class of 2{D} tropical
  climate model.
\newblock {\em Math. Methods Appl. Sci.}, 42(8):2533--2543, 2019.

\bibitem{MR3479523}
Jinkai Li and Edriss Titi.
\newblock Global well-posedness of strong solutions to a tropical climate
  model.
\newblock {\em Discrete Contin. Dyn. Syst.}, 36(8):4495--4516, 2016.

\bibitem{MR3825173}
Ming Li and Haifeng Shang.
\newblock Large time decay of solutions for the 3{D} magneto-micropolar
  equations.
\newblock {\em Nonlinear Anal. Real World Appl.}, 44:479--496, 2018.

\bibitem{MR1041744}
Grzegorz Lukaszewicz.
\newblock On the existence, uniqueness and asymptotic properties for solutions
  of flows of asymmetric fluids.
\newblock {\em Rend. Accad. Naz. Sci. XL Mem. Mat. (5)}, 13(1):105--120, 1989.

\bibitem{MR1711268}
Grzegorz Lukaszewicz.
\newblock {\em Micropolar fluids}.
\newblock Modeling and Simulation in Science, Engineering and Technology.
  Birkh\"{a}user Boston, Inc., Boston, MA, 1999.
\newblock Theory and applications.

\bibitem{MR767409}
Ky\={u}ya Masuda.
\newblock Weak solutions of {N}avier-{S}tokes equations.
\newblock {\em Tohoku Math. J. (2)}, 36(4):623--646, 1984.

\bibitem{NichePerusato2020}
C\'esar~J. Niche and Cilon~F. Perusato.
\newblock Sharp decay estimates and asymptotic behaviour for 3d
  magneto-micropolar fluids, (available in: arxiv.org/abs/2006.14427).

\bibitem{MR3355116}
C\'esar~J. Niche and Mar\'\i a~E. Schonbek.
\newblock Decay characterization of solutions to dissipative equations.
\newblock {\em J. Lond. Math. Soc. (2)}, 91(2):573--595, 2015.

\bibitem{MR1749867}
Marcel Oliver and Edriss~S. Titi.
\newblock Remark on the rate of decay of higher order derivatives for solutions
  to the {N}avier-{S}tokes equations in {${\bf R}^n$}.
\newblock {\em J. Funct. Anal.}, 172(1):1--18, 2000.

\bibitem{MR1810322}
Elva~E. Ortega-Torres and Marko~A. Rojas-Medar.
\newblock Magneto-micropolar fluid motion: global existence of strong
  solutions.
\newblock {\em Abstr. Appl. Anal.}, 4(2):109--125, 1999.

\bibitem{MR1484679}
Marko~A. Rojas-Medar.
\newblock Magneto-micropolar fluid motion: existence and uniqueness of strong
  solution.
\newblock {\em Math. Nachr.}, 188:301--319, 1997.

\bibitem{MR571048}
Mar\'{\i}a~E. Schonbek.
\newblock Decay of solutions to parabolic conservation laws.
\newblock {\em Comm. Partial Differential Equations}, 5(7):449--473, 1980.

\bibitem{MR837929}
Maria~E. Schonbek.
\newblock Large time behaviour of solutions to the {N}avier-{S}tokes equations.
\newblock {\em Comm. Partial Differential Equations}, 11(7):733--763, 1986.

\bibitem{MR1312701}
Maria~E. Schonbek.
\newblock Large time behaviour of solutions to the {N}avier-{S}tokes equations
  in {$H^m$} spaces.
\newblock {\em Comm. Partial Differential Equations}, 20(1-2):103--117, 1995.

\bibitem{MR1396285}
Maria~E. Schonbek and Michael Wiegner.
\newblock On the decay of higher-order norms of the solutions of
  {N}avier-{S}tokes equations.
\newblock {\em Proc. Roy. Soc. Edinburgh Sect. A}, 126(3):677--685, 1996.

\bibitem{MR775190}
Maria~Elena Schonbek.
\newblock {$L^2$} decay for weak solutions of the {N}avier-{S}tokes equations.
\newblock {\em Arch. Rational Mech. Anal.}, 88(3):209--222, 1985.

\bibitem{MR3455664}
Renhui Wan.
\newblock Global small solutions to a tropical climate model without thermal
  diffusion.
\newblock {\em J. Math. Phys.}, 57(2):021507, 13, 2016.

\bibitem{MR2158216}
Norikazu Yamaguchi.
\newblock Existence of global strong solution to the micropolar fluid system in
  a bounded domain.
\newblock {\em Math. Methods Appl. Sci.}, 28(13):1507--1526, 2005.

\end{thebibliography}
\end{document}